\documentclass[10pt]{amsart}
\usepackage{color}
\usepackage{hyperref}
\usepackage{scalerel}
\usepackage{amsmath,amssymb,amsthm,amsfonts,mathrsfs,amsopn}
\usepackage{mathtools}
\mathtoolsset{showonlyrefs}
\usepackage[latin1]{inputenc}
\usepackage[active]{srcltx}
\usepackage{graphicx}
\usepackage{exscale,relsize}
\usepackage{epsfig,graphics}
\usepackage{psfrag}
\usepackage{caption}
\usepackage{subcaption}
\usepackage[toc,page]{appendix}
\usepackage{bigints}
\usepackage{esint}
\usepackage[all]{xy}
\usepackage[margin=0.9in]{geometry}
\newcommand{\parenthesis}[1]{\left(#1\right)} 
\newcommand{\braces}[1]{\left\{#1\right\}} 

\newcommand{\R}{\mathbb{R}}

\newcommand{\dd}{\mathop{}\!\mathrm{d}}
\newcommand{\eps}{\varepsilon}
\newcommand{\p}{\partial}

\DeclareMathOperator{\dist}{dist}

\def\N{\mathbb{N}}
\def\R{\mathbb{R}}
\def\m1{{I\!\!M}}

\def\ee{\`e}

\def\aa{\`a}

\renewcommand{\to}{\rightarrow}
\newcommand{\pa}{\partial}

\newcommand{\ov}[1]{\overline{#1}}
\newcommand{\un}[1]{\underline{#1}}

\newcommand{\intbar}{\mathop{\int\makebox(-15.5,0){\rule[6pt]{.7em}{0.3pt}}\kern-6pt}\nolimits}

\newcommand{\ii}{\infty}

\newcommand{\dt}{\delta}

\newcommand{\al}{\alpha}

\newcommand{\om}{\Omega}
\newcommand{\lm}{\lambda}

\renewcommand{\le}{\leq}

\newcommand{\fbi}{{$(F_{I})$}}
\newcommand{\prl}{{$(P_{\lm})$}}

\newtheorem{theorem}{Theorem}[section]
\newtheorem{proposition}[theorem]{Proposition}
\newtheorem{lemma}[theorem]{Lemma}
\newtheorem{corollary}[theorem]{Corollary}
\newtheorem{remark}[theorem]{Remark}
\newtheorem{definition}[theorem]{Definition}
\newcommand{\brm}{\begin{remark}\rm}
\newcommand{\erm}{\end{remark}}
\newcommand{\bdf}{\begin{definition}\rm}
\newcommand{\edf}{\end{definition}}
\newcommand{\bte}{\begin{theorem}}
\newcommand{\ete}{\end{theorem}}
\newcommand{\bpr}{\begin{proposition}}
\newcommand{\epr}{\end{proposition}}
\newcommand{\ble}{\begin{lemma}}
\newcommand{\ele}{\end{lemma}}
\newcommand{\bco}{\begin{corollary}}
\newcommand{\eco}{\end{corollary}}
\newcommand{\beq}{\begin{equation}}
\newcommand{\eeq}{\end{equation}}
\newcommand{\bdm}{\begin{displaymath}}
\newcommand{\edm}{\end{displaymath}}

\newcommand{\graf}[1]{\left\{\begin{array}{ll}#1\end{array}\right.}

\def\sideremark#1{\ifvmode\leavevmode\fi\vadjust{\vbox to0pt{\vss
 \hbox to 0pt{\hskip\hsize\hskip1em \vbox{\hsize2.1cm\tiny\raggedright\pretolerance10000 \noindent #1\hfill}\hss}\vbox to15pt{\vfil}\vss}}}

\begin{document}
\numberwithin{equation}{section}
\parindent=0pt
\hfuzz=2pt
\frenchspacing

\title[Dancer-Yan Spikes revisited]{Classification of singular limits for free boundary and singularly perturbed elliptic problems:
the Dancer-Yan spikes revisited.}


\author[D. Bartolucci]{Daniele Bartolucci}
\address{Daniele Bartolucci, Department of Mathematics, University of Rome \emph{"Tor Vergata"}, Via della ricerca scientifica n.1, 00133 Roma.}
\email{bartoluc@mat.uniroma2.it}

\author[A. Jevnikar]{Aleks Jevnikar}
\address{Aleks Jevnikar, Department of Mathematics, Computer Science and Physics, University of Udine, Via delle Scienze 206, 33100 Udine, Italy.}
\email{aleks.jevnikar@uniud.it}

\author[J. Wei]{Juncheng Wei}
\address{Juncheng Wei, Department of Mathematics, Chinese University of Hong Kong, Shatin, Hong Kong.}
\email{wei@math.cuhk.edu.hk}

\author[R. Wu]{Ruijun Wu}

\address{Ruijun Wu, School of mathematics and statistics, Beijing Institute of Technology, Zhongguancun South Street No. 5, 100081 Beijing, P.R.China.}
\email{ruijun.wu@bit.edu.cn}

\thanks{ D.B.
is partially supported by the MIUR Excellence Department Project MatMod@TOV
awarded to the Department of Mathematics, University of Rome ``Tor Vergata'' and by PRIN project 2022, ERC PE1\_11,
``{\em Variational and Analytical aspects of Geometric PDEs}'' and by the E.P.G.P. Project sponsored by the University of Rome ``Tor Vergata''. \\
 \indent A.J. is partially supported by INdAM-GNAMPA project ``{\em Analisi qualitativa di problemi differenziali non lineari}'' and PRIN Project 20227HX33Z ``{\em Pattern formation in nonlinear phenomena}''.\\
\indent D.B. and A.J. are member of the INDAM Research Group ``Gruppo Nazionale per l'Analisi Matematica, la Probabilità e le loro Applicazioni''.\\
\indent J.W. is partially supported by GRF fund of RGC of Hong Kong entitled
``\emph{New frontiers in singularity formations of nonlinear partial differential equations}''.}

\begin{abstract}
We classify the singular limits relative to a free boundary problem arising in plasma physics in dimension $d=2$,
under suitable natural integral bounds.
It turns out that one of the asymptotic behaviors allowed corresponds to the Dancer-Yan spikes (J. London Math. Soc. ({\bf 78}) 2008, 639--662).
Interestingly enough, roughly speaking  and unlike the higher dimensional case, it is not true that any solution in the limit is a Dancer-Yan spike.
Indeed, the spiking structure is more rich and we succeed in a detailed description
of the singular behavior by a careful analysis, from local to global, of the tiny difference between the maximum value of the spikes and their
``vanishing level'' defining the free boundary.
\end{abstract}
\maketitle



\
\

\setcounter{section}{0}
\setcounter{equation}{0}

\section{Introduction}
Let $\Omega\subset \R^2$ be a bounded domain of class~$C^{2,\beta}$ for some~$\beta\in (0,1)$, and
let $p\in (1,+\infty)$. For any fixed~$I>0$, we consider the pairs~$({\rm v},\gamma)\in C^{2,\beta}(\Omega)\cap C^0(\overline{\Omega})\times \R$
solving the free boundary problem:
\begin{align}\tag{$F_I$}\label{fbi}
 \begin{cases}
  -\Delta {\rm v} = [{\rm v}]_+^p, & \mbox{ in } \Omega, \vspace{2mm}\\
  \quad\;\; {\rm v} = \gamma, & \mbox{ on } \partial\Omega, \vspace{2mm}\\
  \bigints\limits_{\Omega} [{\rm v}]_+^p = I
 \end{cases}
\end{align}
This system is a simplified model of the Grad-Shafranov equation, describing plasma equilibria in a section of a Tokamak.
The interest about this equation has been recently renewed (see \cite{far,jey, pent1, pent2} and references quoted therein),
due to its relevance to the seek of cold fusion (\cite{Frei,Stac}). The so called ``plasma region'', is by definition the set $\om_+=\{v>0\}$ which is
the region of the Tokamak occupied by the plasma.
The major issue of cold fusion is to confine the hot plasma far away from the
boundary of the Tokamak, which is why it is interesting to try to describe the set of solutions such that $\om_+\Subset \om$, in which case
the boundary of $\om_+$ is by definition the ``free boundary'' of the plasma. A lot of work has been done to understand
existence, uniqueness, multiplicity of solutions and existence/non-existence/structure of the free boundary of \fbi\, either for
$p=1$ (\cite{CF,CPY,dam,FW,FL,Gal,KNS,KS,Kor3,Pudam,Te,Te2}) or in higher dimension (\cite{AmbM,BeBr,LP,Ort,Sij,Sij1,suz},
\cite{W}-\cite{wol} and more recently \cite{BJW}). As far as we are concerned exactly with the equation \!\fbi, that is the ``plasma problem'' in
dimension $2$ and $p\in (1,+\ii)$, we have the well known results
in \cite{AmbM}, \cite{BaMa}-\cite{BSp}, \cite{BeBr}, \cite{DY}, \cite{Liu,Mar} and more recently \cite{BJ1}, \cite{BJ3}, \cite{BHJY}.
Due to the fundamental results of Beresticky-Brezis (\cite{BeBr}), it is well known that for any~$I>0$
there exists at least one solution of~\fbi.

\medskip

We wish to make a first step in the classification of the singular limits allowed for solutions of \fbi\, as~$I\to +\ii$, a situation in which one expects to find
a non trivial free boundary structure. Indeed, among other results cited above, in dimension $d\geq 3$ and $1<p<\frac{d}{d-2}$,
the existence and qualitative properties of solutions of \fbi\, in the limit $I\to +\ii$, sharing a nice free boundary structure (sometimes called ``spikes'') was proved in \cite{FW,W}.
In a recent paper~\cite{BJW}, it has been shown that, under suitable natural integral bounds, in the limit $I\to +\ii$ any solution is the glueing of finitely many
spikes in the sense of \cite{FW,W}. We would like to obtain a similar result in dimension $d=2$ but we face a genuinely new difficulty. In fact, the construction
in \cite{FW,W} as well as the results in \cite{BJW} crucially rely on the classification of solutions of
$-\Delta w=[w-1]_+^p \mbox{ in } \R^d$, $\int\limits_{\R^d}[w-1]_+^p<+\ii$, $w>0, w\to 0$, $|x|\to +\ii$. However, as first noticed in \cite{DY},
there is no solution of this problem in dimension $d=2$, which prevents one to adopt the argument in \cite{FW} and \cite{W}.
This is why the results in \cite{FW} has been extended somehow to the case $d=2$ in \cite{DY} in a different way,
based on a careful choice of the solution of the limiting global problem and a simultaneous rescaling and blow up of the solutions.
 The corresponding \emph{model} solutions, which we refer to as the ``Dancer-Yan spikes'', are described in subsection \ref{subs:3.1} below.
First of all we classify the \emph{finite mass} solutions of the corresponding planar equations,
see Proposition \ref{prop:entire solution} about problem \eqref{eq:entire eqn}. Surprisingly enough,
it turns out that, unlike the case $d\geq 3$, in dimension $d=2$ \underline{it is not true} that, under suitable natural integral bounds,
in the limit $I\to +\ii$ any solution is the glueing of finitely many Dancer-Yan spikes (\cite{DY}). We illustrate this point below in terms of
more natural variables as in \cite{BJW}.

\medskip

For any fixed $p>1$ and $I>0$ let us set $I=\lambda^q$, where~$q=\frac{p}{p-1}$ is the H\"older conjugate to~$p$. We
consider the new variables~$(\alpha,\psi)\in \R\times C^{2,\beta}(\overline{\Omega})$ defined as follows
\begin{align}\label{eq:change variable}
   \begin{cases}
    \gamma= \lambda^{\frac{1}{p-1}}\alpha, \vspace{2mm}\\
    {\rm v} = \lambda^{\frac{1}{p-1}}(\alpha+\lambda\psi).
   \end{cases}
  \end{align}
  Thus, as far as $p>1$, for fixed $I>0$, ~\eqref{fbi} is equivalent to the following problem for fixed $\lm=I^{\frac1q}$,
\begin{align}\tag{$P_\lambda$}\label{prl}
 \begin{cases}
  -\Delta \psi = [\alpha+\lambda\psi]_+^p, & \mbox{ in } \Omega, \vspace{2mm} \\
\quad\;\; \psi=0, & \mbox{ on } \partial \Omega, \vspace{2mm}\\
  \bigints\limits_{\Omega} [\alpha+\lambda\psi]_+^p =1.
 \end{cases}
\end{align}
Besides the equivalence with \fbi, problems of the form \prl\, arise in other contexts but with different boundary conditions (\cite{Hor}).
Actually, \prl\, is closely related to the well known class of singularly perturbed problems (\cite{DY,FW,W}) as it is readily seen
putting
$$
v=\frac{\lambda}{|\alpha|}\psi \quad \mbox{ and }\quad  \varepsilon\coloneqq (|\alpha|^{p-1}\lambda)^{-1/2},
$$
which, in the limit $\alpha\to -\ii$, satisfy (see \cite{BJW} for a proof) $\lm\to +\ii$, $\varepsilon\to 0$ and
\begin{align}\label{eq:v form}
    \begin{cases}
        -\varepsilon^2\Delta v= [v-1]_+^p, & \mbox{ in } \Omega, \vspace{2mm} \\
        \quad\quad\;\; v=0, & \mbox{ on } \partial \Omega, \vspace{2mm} \\
        \bigint\limits_{\Omega} [v-1]_+^p=\frac{1}{|\alpha|^p}\to 0.
    \end{cases}
\end{align}

In order to classify the asymptotic behavior of solutions, we drop at first boundary conditions in \prl\, and
consider a sequence of solutions of
\begin{align}\label{eq:blow up eqn}
    \begin{cases}
        -\eps_n^2\Delta v_n= [v_n-1]_+^p, & \mbox{ in } \Omega\vspace{2mm}\\
        v_n\geq 0, & \mbox{ in } \Omega
    \end{cases}
\end{align}
with~$\eps_n\searrow 0$, which satisfy:
\begin{align}\label{Hypo on p-1} \tag{{\bf H1}}
    \boxed{\int\limits_\Omega \frac{1}{\eps_n^2}[v_n-1]_+^{p-1}\dd{x} \leq H_{p-1}, }
\end{align}
for some fixed $H_{p-1}>0$. It can be shown, see Proposition \ref{prop:local vanishing} below, that a sequence of solutions of \eqref{eq:blow up eqn}
satisfying \eqref{Hypo on p-1} also satisfies $[v_n-1]_+\to 0$ locally uniformly in $\om$. It is worth to remark that this
is a major difference with the higher dimensional case (see \cite{FW}, \cite{W} and \cite{BJW}) where local maxima stay bounded below far away from
$1$. As a consequence, in particular a subtle problem arise which is to analyze the fine properties of the vanishing quantity
$[v_n-1]_+$ around a local maximizer. To this aim we also assume that $v_n$ satisfies,
\begin{align}\label{Hypo on p} \tag{{\bf H2}}
\boxed{\int\limits_\Omega \frac{\theta_n}{\eps_n^2}[v_n-1]_+^{p}\dd{x} \leq H_p, }
\end{align}
for some $H_{p}>0$, where, here and in the rest of this work, $\theta_n$ is defined as in \cite{DY}: first define~$s_n$ by
\begin{align}\label{thetan1}
        \parenthesis{\frac{\eps_n}{s_n}}^{\frac{2}{p-1}} \phi'(1)\ln(\sqrt{\pi} s_n)=1,
    \end{align}
    where $\phi$ is the unique solution of the Emden equation \eqref{eq:Emden}; and then put
    \begin{align}\label{thetan2}
        \theta_n\coloneqq \phi'(1)\ln(\sqrt{\pi}s_n).
    \end{align}
    Remark that for~$\eps_n$ sufficiently small $s_n$ is uniquely defined and~$s_n\to 0^+$ as $\eps_n\to 0^+$.
    We assume w.l.o.g. that $s_n\leq \frac{1}{\sqrt{\pi}}$.
    Note that since~$\phi'(1)<0$, we have $\lim\limits_{n\to+\infty}\theta_n=+\infty$.

\medskip

Here and in the rest of this notes we will often pass to subsequences which will not be relabelled.
In the following statement $I_{p-1},I_p$, are
defined in \eqref{Ip-1}, \eqref{Ip}, $s_n\to0$ is defined in \eqref{thetan1} and $w^*$ is defined in \eqref{wstar}. Then we have,
    \begin{theorem}\label{thm:ALT near maximum}
        Let~$v_n$ be a sequence of solutions of~\eqref{eq:blow up eqn} satisfying~\eqref{Hypo on p-1},\eqref{Hypo on p}.
        Then:
        \begin{itemize}
            \item[either] {\rm (A) [Vanishing]} For any~$\Omega_0\Subset \Omega $ there exists~$n_0\in\mathbb{N}$ and~$C_0>0$,
            both depending by~$\Omega_0$, such that,
            \begin{align*}
                [v_n-1]_+ = 0, \quad \mbox{ in } \Omega_0, \qquad \forall n\geq n_0,
            \end{align*}

            $\hspace{-4mm}$in which case, for any~$\Omega_1$ satisfying~$\Omega_0\Subset \Omega_1\Subset \Omega$, there exists~$C(\Omega_0,\Omega_1)>0$ such that
            \begin{align}\label{eq:Vanishing-interior est}
                \|v_n\|_{L^\infty(\Omega_0)} \leq C(\Omega_0,\Omega_1) \int\limits_{\Omega_1} v_n\dd{x}, \qquad \forall n\geq n_0,
            \end{align}

            \item[or$\quad$] {\rm (B) } There exists an open subdomain~$\Omega_0\Subset \Omega$ and sequence of points~$x_n$ in~$\Omega_0$ such that\\
            $x_n\to x^*\in \Omega_0$ and
            \begin{align*}
                v_n(x_n)=\sup_{\Omega} v_n >1,
            \end{align*}

            $\hspace{-4mm}$in which case, setting
            $$
            t_n:= \braces{\frac{\phi(0)}{\theta_n(v_n(x_n)-1)}}^{\frac{p-1}{2}}
            $$
            and~$\tilde{v}_n$ as follows,
            \begin{align}\label{eq:refined rescaling}
        \tilde{v}_n(y)\coloneqq \phi'(1)\ln(\sqrt{\pi}s_n) \parenthesis{ v_n(x_n+ s_n y)-1}
        =\theta_n\parenthesis{ v_n(x_n+ s_n y)-1}
           \end{align}
            the sequence of \textbf{normalized} functions,
            \begin{align}\label{eq:un def}
                u_n(z)= {\mathcal{R}}_{t_n}\tilde{v}_n(z):= t_n^{\frac{2}{p-1}}\tilde{v}_n(t_n z)
            \end{align}
            satisfies~$u_n(0)=\phi(0)$ and
            \begin{align}\label{eq:un eqn}
                -\Delta u_n(z) = [u_n(z)]_+^p, \qquad \forall z\in \R^2, |z|<\frac{\dist(x^*,\p\Omega)}{4s_n t_n}.
            \end{align}
            Then either~$s_n t_n \to 0^+$, which means that there exists a \underline{spike}, or else~$s_n t_n \geq C >0$, which means
            that the spike is \underline{fading out}.
            More exactly, we have the following alternatives:

            \begin{itemize}
                \item[either] {\rm (B-i) [Type I spikes]} $t_n\to t_\infty\in (T_0,+\infty)$, where~$T_0^{\frac{2}{p-1}}=\frac{1}{2}\frac{H_p}{I_p}>0$,

                and then there exists a sequence~$R_n\to +\infty$ such that~$R_{n}s_n t_n\to 0^+$, and
                    \begin{align*}
                        \|u_n - w^*\|_{C^2(B_{2R_n}(0))}\to 0, \quad \mbox{ as } n\to+\infty.
                    \end{align*}
                    Furthermore,
                    \begin{align*}
                        \lim_{n\to+\infty} \frac{1}{\eps_n^2}\int\limits_{B_{2R_n s_n t_n}(x_n)} [v_n-1]_+^{p-1}\dd{x}=I_{p-1}, & &
                        \lim_{n\to+\infty}\frac{\theta_n}{\eps_n^2}\int\limits_{B_{2R_n s_n t_n}(x_n)} [v_n-1]_+^{p}\dd{x}=\frac{I_p}{t_\infty^{\frac{2}{p-1}}},
                    \end{align*}

                \item[or$\quad$] {\rm (B-ii) [Type II spikes]} $t_n\to+\infty\equiv t_\infty$ but~$s_n t_n\to 0^+$,

                and then there exists a sequence~$R_n\to +\infty$ such that~$R_n s_n t_n\to 0^+$,
                 \begin{align*}
                        \|u_n - w^*\|_{C^2(B_{2R_n}(0))}\to 0, \quad \mbox{ as } n\to+\infty,
                    \end{align*} and
                    \begin{align*}
                        \lim_{n\to+\infty} \frac{1}{\eps_n^2}\int\limits_{B_{2R_n s_n}(x_n)} [v_n-1]_+^{p-1}\dd{x}=I_{p-1}, & &
                        \lim_{n\to+\infty} \frac{\theta_n}{\eps_n^2}\int\limits_{B_{2R_n s_n}(x_n)} [v_n-1]_+^{p}\dd{x}=0,
                    \end{align*}

                \item[or$\quad$] {\rm (B-iii) [Fading spikes]} $t_n\to +\infty$ and~$s_n t_n \geq C>0$, \mbox{\rm (}which includes the case
                $s_n t_n \to+\infty$, that is $v_n(x_n)\to 1$ super fast\mbox{\rm )},
                and then~$v_n(x)-1$ decays at least as fast as $\eps_n^{\frac{2}{p-1}}$ in a \un{full disk} around $x^*$.
                More exactly for any $\dt\leq \frac{1}{4}\dist(x^*,\partial\Omega)$ there exists a bounded function $f\in L^{\ii}(B_\dt(x^*))$ such that
                $f(x)\leq \phi(0)$ and
                \begin{align*}
                    v_n(x)= 1+ \frac{\eps_n^{\frac{2}{p-1}}}{(s_n t_n)^{\frac{2}{p-1}}}(f(x)+o(1)), \quad \forall x \colon |x-x^*|<\dt,
                \end{align*}
and
                \begin{align}\nonumber
                        &\frac{1}{\eps_n^2}\int_{B_\dt(x_n)} [v_n(x)-1]_+^{p-1} \dd{x}\leq \pi \dt^2\frac{\phi^{p-1}(0)}{(s_n t_n)^{2}}(1+o(1)) , \\
                        \label{eq:Fading energies}
                        &\frac{\theta_n}{\eps_n^2}\int\limits_{B_{\dt}(x_n)} [v_n-1]_+^{p}\dd{x}\leq
                        \pi \dt^2 \frac{\phi^{p}(0)}{(s_n t_n)^{\frac{2p}{p-1}}} s_n^{\frac{2}{p-1}}(1+o(1)).
                    \end{align}
            \end{itemize}
        \end{itemize}
    \end{theorem}
Although Theorem \ref{thm:ALT near maximum} describes the singular behavior just around one sequence of maximum points,
due to the richness of the picture it seems worth to present it in this form. A naive graphical representation of the spikes
is provided in Figure 1 at the end of section \ref{sect:blowup max}.
We refer to Remark \ref{rem:4.6} below for an equivalent formulation of the alternatives in terms of the $\tilde{v}_n$ variables.
In any case, as mentioned above, the spikes of Type I are the Dancer-Yan spikes, while the result shows
that we could have in principle other two singular behaviors, at least as far as \eqref{Hypo on p-1}, \eqref{Hypo on p} are satisfied.
The crux of the argument is
to realize that the $\tilde{v}_n$ variables in \eqref{eq:refined rescaling} are not well suited to describe all the profiles allowed by the singular limit,
a new rescaling being needed (i.e. \eqref{eq:un def}) which uses in a careful way the invariance of the equation, see \eqref{rinv1}-\eqref{rinv3}.\\

Remark that
\eqref{Hypo on p-1} is crucial as it prevents in the limit ``infinite mass'' solutions of the planar problem \eqref{eq:entire eqn} in case (B-ii),
in particular providing a control on the $(p-1)$-mass, i.e. the quantity $\frac{1}{\eps_n^2}\int\limits_{\om} [v_n-1]_+^{p-1}\dd{x}$.
Far from being a technical point, infinite mass solutions of \eqref{eq:entire eqn} exist, see subsection \ref{subs:3.3}.
The assumption \eqref{Hypo on p} ensures the same property in case (B-i),
in particular providing a control on the $(p)$-mass, i.e. the quantity $\frac{\theta_n}{\eps_n^2}\int\limits_{\om} [v_n-1]_+^{p}\dd{x}$.\\
It is natural to wonder whether or not we can drop \eqref{Hypo on p}, which in fact in some situations seems to be redundant.
However, we still don't know the answer to this natural question.\\

It is worth to remark that only the $(p-1)$-mass
seems to be well suited to satisfy a quantization phenomenon in the same spirit of critical equations (Yamabe $d\geq 3$, Liouville $d\geq 2$),
while this is clearly not the case for the  $p$-mass which is not quantized in general, see Theorem \ref{thm:ALT many maximum} below.

\bigskip

At this point a major problem arise in the description of the singular limit in case of multiple-spiking at some point, that is,
whenever many spikes of different type are found to
be clustering at the same interior point, say $x_{n,j}\to x^*\in\Omega$. Remark that solutions of this sort exist (see \cite{DY}).
However, a full description of this phenomenon would require the analysis of the local
interaction of the three different spikes, i.e. Type I, Type II and Fading, which is rather hard in general. Unfortunately we still miss a
full description of this singular limit, which would play in this context the same role say of the local concentration-compactness theory for Liouville
type equations (\cite{bm,ls}).
Nevertheless we succeed in obtaining a full classification of the clustering
at an interior point adding back Dirichlet boundary conditions. In fact, in this situation, by a suitable non-vanishing assumption, we are able to rule out
both Vanishing and Fading. To this aim, let us associate to the sequence~$v_n$ the following sets:
    \begin{align*}
        \Sigma_{\rm I}\coloneqq \braces{x\in\Omega \mid \exists\, x_{n}\to x\; \mbox{ s.t. } \; \mbox{ the sequence }\;
        v_n(x_n)\; \mbox{ yields a spike of Type I }},
    \end{align*}
    \begin{align*}
        \Sigma_{\rm II}\coloneqq \braces{x\in\Omega \mid \exists\, x_{n}\to x\; \mbox{ s.t. } \; \mbox{ the sequence }\;
        v_n(x_n)\; \mbox{ yields a spike of Type II }},
    \end{align*}
    and finally define the \un{singular set},
    \begin{align*}
        \Sigma\coloneqq \Sigma_{\rm I} \cup \Sigma_{\rm II}.
    \end{align*}
    For any~$r>0$, let
    \begin{align*}
        (\Sigma)_r\coloneqq \braces{x\in\Omega\mid \dist(x,\Sigma)<r}
    \end{align*}
    be the~$r$-neighborhood of~$\Sigma$ and $\N$ denote the set of non negative integers.  Then we have,

    \begin{theorem}\label{thm:ALT many maximum}
        Let~$v_n$ be a sequence of solutions of~\eqref{eq:blow up eqn} satisfying $v_n|_{\p\Omega}=0$ and~\eqref{Hypo on p-1},
        \eqref{Hypo on p}. Assume that the following non-vanishing condition is satisfied, for some~$C_0>0$:
                \begin{align}\label{eq:NV p}\tag{\bf NV$_p$}
                    \frac{\theta_n}{\eps_n^2}\int\limits_{\Omega} [v_n(x)-1]_+^p\dd{x}\geq C_0,\quad \forall\,n\in\N.
                \end{align}
        Then the singular set $\Sigma$ is not empty, it consists of finitely many points, and for any~$r>0$, there exists~$n_r\in\mathbb{N}$ such that
        \begin{align*}
            [v_n(x)-1]_+ =0, \qquad \forall x\in \Omega\setminus (\Sigma)_r, \; \forall n > n_r.
        \end{align*}
        In particular there exist $N_I\in\N$, $N_{II}\in\N$ with $N_{I}\geq 1$ such that, defining
        \begin{align}
         X_{\rm I}=\{x^*_{n,i}\}_{i\in\{1,\cdots,N_I\}, n\in\N}
        \end{align}
        to be the sequences of local maximizers yielding Type I spikes and
        \begin{align}
         X_{\rm II}=\{x^{**}_{n,i}\}_{i\in\{1,\cdots,N_{II}\},n\in\N}
        \end{align}
        to be the sequences of local maximizers yielding Type II spikes, the following
        facts hold true:
         \begin{itemize}
            \item[(a)]
     $$
     v_n(x)= 1+\frac{1}{t_{n,j}^{\frac{2}{p-1}}\theta_n} w^*(\frac{x-x_{n,j}}{s_n t_{n,j}})+ o(\frac{1}{t_{n,j}^{\frac{2}{p-1}}\theta_n }),
     \quad |x-x_{n,j}|\leq R_{n}s_nt_{n,j},\;j=1,\cdots,N_I+N_{II},
     $$
 where
     \begin{align*}
    & t_{n,j}\to t_{\ii,j}\in (T_0,+\ii] \mbox{ as }n\to+\infty,\quad  T_0^{\frac{2}{p-1}}=\frac{1}{2}\frac{H_p}{I_p}>0, \mbox{ and }\\
    & \graf{t_{\ii,j}\in (T_0,+\ii), \;\mbox{if}\;\; x_{n,j}=x^*_{n,i},\mbox{ for some }\,i\in \{1,\cdots,N_{I}\}\\\,\\
     t_{\ii,j}=+\ii, \;\mbox{if}\;\; x_{n,j}=x^{**}_{n,i},\,\mbox{ for some }i\in \{1,\cdots,N_{II}\}}
      \end{align*}
            \item[(b)]  \begin{align*}
  v_n(x)= \frac{1}{\theta_n}\left(\sum_{i=1}^{N_I} \frac{I_p}{t_{n,i}^{\frac{2}{p-1}}}G(x,x^*_{\ii,i})+o_r(1)\right),\;\forall\, x\in \om\setminus (\Sigma)_r
    \end{align*}
where $o_r(1)$ uniformly converges to $0$ for any fixed $r$ small enough;
           \item[(c)]
           \begin{align}\label{p-1:quant}
              \lim\limits_{n\to+\infty}\frac{1}{\eps_n^2}\int\limits_{\om} [v_n(x)-1]_+^{p-1}\dd{x}=
              (N_I+N_{II})I_{p-1}
            \end{align}
            \begin{align}\label{p-2:nquant}
              \lim\limits_{n\to+\infty}\frac{\theta_n}{\eps_n^2}\int\limits_{\om} [v_n(x)-1]_+^{p}\dd{x}\to
              \frac{1}{\gamma_\ii}:=\sum\limits_{j\in \{1,\cdots,N_I\}}\frac{I_p}{t_{\ii,j}^{\frac{2}{p-1}}}.
            \end{align}

            \item[(d)] the plasma region, that is the subset
            $$
            \Omega_{n,+}\coloneqq\braces{x\in\Omega\mid v_n(x)>1}
            $$
            consists of \un{asymptotically round points} in the sense of Caffarelli--Friedman \mbox{\rm (}\cite{CF}\mbox{\rm )},\\
            namely, for any~$0<\theta<1$
    \begin{align}\label{eq:round pts}
        \bigcup_{j=1}^{N_I+N_{II}} B_{(1-\theta) s_n t_{n,j}}(x_{n,j})
        \Subset \Omega^{n,+}
        \Subset  \bigcup_{j=1}^{N_I+N_{II}} B_{(1+\theta) s_n t_{n,j}}(x_{n,j})
    \end{align}
    for any $n$ sufficiently large;

        \item[(e)] let us define,
    \begin{align*}
        \Sigma_{\rm I} =\braces{x^*_{\ii,1}, \cdots, x^*_{\ii,m_1}}
    \end{align*}
    to be the set of {spike points} of Type I, and for each~$\ell=1,2,\cdots, m_1$, denote
    \begin{align}\label{p-mass}
        M_\ell\coloneqq \sum_{i=1,\cdots,N_{I}\; : \; x^*_{n,i}\to x^*_{\ii,\ell}} \frac{1}{t_{\infty,i}^{\frac{2}{p-1}}}.
    \end{align}
    Consider the following Kirchhoff-Routh Hamiltonian
    \begin{align}\label{KRH}
        \mathcal{H}(q_1, \cdots, q_{m_1})\coloneqq
        \sum_{i=1}^{m_1} M_i^2 H(q_i, q_i)
         + \sum_{i\neq l} M_i M_l G(q_i, q_l),
     \end{align}
     then the vector $(x^*_{\ii,1}, \cdots, x^*_{\ii,m_1})$ is a critical point of $\mathcal{H}$.

        \end{itemize}

    \end{theorem}

Remark that (d) is a generalization of the well known result for $p=1$ by Caffarelli-Friedman (\cite{CF}).   
Also, we point out that $\Sigma$ is a set of interior points. The statement does not exclude that~$\Sigma_{\rm I}\cap \Sigma_{\rm II}\neq\emptyset$.
In principle there could be several sequences of local maximizers clustering
at a fixed
point~$x\in \Sigma_{\rm I}\cap \Sigma_{\rm II}$, each one yielding a spike of Type I or II. Recall that we already know that
spikes of Type I are not ``simple'' in general (\cite{DY}). A naive graphical representation of the global behavior of the
spikes is provided in Figure 2 in section \ref{sec6}.\\
We did our best to simplify the exposition of the proof by splitting the first part of the argument into
three subsections. It is easy to prove that any two spikes (Type I, Type II or Fading) converging at the same point (say $x^*$)
cannot be too close each other (see subsection \ref{subs:5.1}).
On the other side, it is more delicate to prove that if one already has a spike of Type I or of Type II at $x^*$,
then there can be no Fading spikes at $x^*$, which is done
in subsection \ref{subs:5.2}. At last, the formation of two spikes either of Type I or of Type II clustering at the same point
is handled in subsection \ref{subs:5.3}.
These are the building blocks of the proof which is then completed in section \ref{sec6}.

\bigskip

Our last result is concerned with the non trivial consequences of Theorem \ref{thm:ALT many maximum} about \prl.

 \begin{theorem}\label{thm:last} Let $\psi_n$ be sequence of solutions of \prl\, for $\lm=\lm_n\to +\ii$, $\al=\al_n\leq -1$  and assume that
 \begin{equation}\label{H3}
 \frac{1}{C_{p}}\leq \frac{\lm_n}{|\al_n|}\log(|\al_n|^{p-1}\lm_n)\leq C_{p},
 \end{equation}
 and
 \begin{equation}\label{H4}
 \lm_n\int\limits_{\om}[\al_n+\lm_n\psi_n]^{p-1}\leq C_{p-1}
 \end{equation}
 for some $C_p>1$, $C_{p-1}>0$. Then the conclusions of Theorem \ref{thm:ALT many maximum} hold true for $v_n=\frac{\lm_n}{|\al_n|}\psi_n$ and in particular,
 recalling
 $\gamma_\ii$ in \eqref{p-2:nquant}, we have that:
 \begin{itemize}
            \item[(i)]
 $$
 |\al_n|=(1+o(1))\gamma_\ii|\phi^{'}(1)|\frac{p-1}{2}\lm_n\log(\lm_n),
 $$
 $$
 \theta_n=(1+o(1))|\phi^{'}(1)|\frac{p-1}{2}\log(|\al_n|)
 $$
 and
 $$
\lm_n\int\limits_{\om}[\al_n+\lm_n\psi_n]^{p-1}\to (N_I+N_{II}) I_p,
 $$

$$
\int\limits_\Omega \frac{\theta_n}{\eps_n^2}[v_n-1]_+^{p}\dd{x} =
\int\limits_\Omega \frac{\lm_n}{|\al_n|}\theta_n[\al_n+\lm_n\psi_n]_+^{p}=\frac{\lm_n}{|\al_n|}\theta_n\to \frac{1}{\gamma_\ii};
$$

            \item[(ii)] with the notations of Theorem \ref{thm:ALT many maximum},
 \eqref{eq:round pts} holds true in the plasma region $$\om_{n,+}=\{\al_n+\lm_n\psi_n>0\}$$  and
         $$
     \psi_n(x)= \frac{|\al_n|}{\lm_n}+(1+o(1))\frac{\gamma_\ii}{t_{n,j}^{\frac{2}{p-1}}} w^*(\frac{x-x_{n,j}}{s_n t_{n,j}})
     \quad \forall\, |x-x_{n,j}|\leq R_{n}s_nt_{n,j},\;j=1,\cdots,N_I+N_{II},
     $$
       where $t_{n,j},\;j=1,\cdots,N_I+N_{II}$ satisfy {\rm (a)} of Theorem \ref{thm:ALT many maximum} and in particular,
         $$
         \frac{|\al_n|}{\lm_n}\leq \psi_n(x)\leq \frac{|\al_n|}{\lm_n}+\frac{\gamma_\ii}{T_{0}^{\frac{2}{p-1}}}\phi(0),\quad \forall\,x\in \om_{n,+}
         $$
         and
         \begin{align}\label{green:rep}
  \psi_n(x)= {(1+o(1))}\gamma_\ii\left(\sum_{i=1}^{N_I} \frac{I_p}{t_{n,i}^{\frac{2}{p-1}}}G(x,x^*_{\ii,i})+o_r(1)\right),\;\forall\,
  x\in \om\setminus (\Sigma)_r
    \end{align}
where $o_r(1)$ uniformly converges to $0$ for any fixed $r$ small enough;\\

            \item[(iii)] {\rm (e)} of Theorem \ref{thm:ALT many maximum} holds, that is,
            the vector of spike points of Type I is a critical point of the Kirchhoff-Routh Hamiltonian \eqref{KRH}.
\end{itemize}
 \end{theorem}

It is easy to see by the proof that the assumption $\al_n\leq -1$ could be replaced by any other bound of the form $\al_n\leq \ov{\al}$, for some
fixed $\ov{\al}<0$.

\medskip

\subsection{Open problems and conjectures.}$\,$\\
We list hereafter some natural open problems about the results discussed so far.\\

Property (e) in Theorem \ref{thm:ALT many maximum} is at hand since, after a suitable rescaling, Type I spikes converge to a sum of Green functions,
essentially by the same mechanism occurring for Liouville-type equations (\cite{yy},\cite{MaW}) or either in dimension $d\geq 3$ (\cite{BJW}).
On the other side, in sharp contrast with the higher dimensional case, the $(p)$-mass \eqref{p-mass} is not
quantized (see also \eqref{p-2:nquant}), while the $(p-1)$-mass does, see \eqref{p-1:quant}.
Remark that \eqref{green:rep} in principle contains terms due to Type II spikes as well, but
$[v_n-1]_+$ is in that case so small that the contribution to the solution is of minor order, see also \eqref{green:1} below.
Therefore Type II spikes are more difficult to analyze.
This is why it seems an interesting open problem to describe the singular limit in case \ref{eq:NV p} fails,
that is $\frac{\theta_n}{\eps_n^2}\int\limits_{\Omega} [v_n(x)-1]_+^p\dd{x}\to 0$, while the $(p-1)$-mass
\eqref{p-1:quant} stays bounded below away from zero.
Remark that the $(p-1)$-mass does not vanish in general for Fading spikes, which makes the problem rather intriguing.
However, in view of the uniqueness result in \cite{BSp}, it is readily seen from the model solutions in section \ref{subs:3.1} that on a disk only
Type I spikes exist. In other words, if \ref{eq:NV p} fails, for $\eps_n$ small there are no solutions at all of
\eqref{eq:blow up eqn} satisfying $v_n|_{\pa\om}=0$ and \eqref{Hypo on p-1} on a disk. This fact suggests that the existence of
Type II spikes could depend by the geometry of the domain, which motivates the following:\\

{\bf Conjecture}. {\it Let $v_n$ be a sequence of solutions of \eqref{eq:blow up eqn} satisfying $v_n|_{\pa\om}=0$ and \eqref{Hypo on p-1}
on a convex domain. If $\frac{\theta_n}{\eps_n^2}\int\limits_{\Omega} [v_n(x)-1]_+^p\dd{x}\to 0$, then for $\eps_n$ small enough there is no
solution of \eqref{eq:blow up eqn}, that is, no Type II spikes exist in this case}.\\

These facts will be discussed in other works.\\

Another interesting open problem is to understand whether or not we could really have multiple spikes of Type I - Type II
clustering at the same point. In fact multiple spikes of Type I can cluster at the same point as shown in \cite{DY}.
This is in sharp contrast with the higher dimensional case where it has been recently proved (see \cite{CM})
that spikes in the sense of \cite{FW} and \cite{W}  are always simple, i.e. one and only
one locally maximizing sequence is attached to each spike point.\\

It could be also interesting describe the set of solutions of the planar equation in problem \eqref{eq:entire eqn} which \un{do not} satisfy the integral
bound therein. A class of ``infinite mass'', one dimensional solutions of this sort is defined in subsection \ref{subs:3.3}.\\

A last comment is needed about the physical meaning of \eqref{H3}, \eqref{H4}. The assumption \eqref{H3} comes out directly from
\eqref{Hypo on p} and \eqref{eq:NV p} and is just constraining the asymptotic behavior of the total current in the Tokamak, which is proportional to
$\int\limits_{\Omega} [v_n(x)-1]_+^p\dd{x}$. It would be rather natural to expect some non trivial spike structure arising from
the control of the corresponding energy, which would require some assumption about $\int\limits_{\Omega} [v_n(x)-1]_+^{p+1}\dd{x}$.
It is rather surprising that instead the seemingly natural assumption at this stage is \eqref{Hypo on p-1}, which yields a control
about $\int\limits_{\Omega} [v_n(x)-1]_+^{p-1}\dd{x}$ whose physical interpretation seems to be unclear. Since
$\mathcal{I}(v)=[v-1]^p_+$ is proportional the current density, we just know that $[v-1]^{p-1}_+$ is proportional to the
variation of $\mathcal{I}(v)$ with respect to $v$.\\

\medskip

This paper is organized as follows. In section \ref{sec2} we discuss some preliminary results together with the classification of solutions of
\eqref{eq:entire eqn} (see Proposition \ref{prop:entire solution}). In section \ref{sec3} we discuss the model profiles of Dancer and Yan and the
one dimensional solutions of \eqref{eq:entire eqn}. In section \ref{sect:blowup max} we prove Theorem \ref{thm:ALT near maximum}, while sections
\ref{sec5} and \ref{sec6} are devoted to the proof of Theorem \ref{thm:ALT many maximum}. In section \ref{sec7} we prove Theorem \ref{thm:last}.

\bigskip

\section{Preliminaries}\label{sec2}

Here we collect some useful facts which will be frequently used later on.

\subsection{Emden solution}

    Let~$B_1(0)\subset\R^2$ denote the unit disk and~$p\in (1,+\infty)$.
    There exists a unique solution (\cite{gnn}) to the Emden equation
    \begin{align}\label{eq:Emden}
    \begin{cases}
         -\Delta\phi=\phi^p, & \mbox{ in } B_1(0), \vspace{2mm} \\
        \phi = 0, & \mbox{ on } \partial B_1(0),
    \end{cases}
    \end{align}
    which is well known (\cite{gnn}) to be radial and radially decreasing.
    Throughout this paper this unique solution we will be just denoted by~$\phi$.
    Furthermore, we denote
    \begin{align}\label{Ip-1}
        I_{p-1}& \coloneqq \int_{B_1(0)}\phi^{p-1}\dd{x}, \\ \label{Ip}
        I_{p} & \coloneqq \int_{B_1(0)} \phi^p \dd{x}=\int_{B_1(0)} -\Delta \phi\dd{x}=-2\pi\phi'(1), \\ \nonumber
        I_{p+1} & \coloneqq \int_{B_1(0)} \phi^{p+1}\dd{x}=\int_{B_1(0)}|\nabla\phi|^2\dd{x},
    \end{align}
    and call them the~$(p+k)$-integrals, for~$k=-1,0,+1$ respectively.
    We list below few useful relations among these quantities:
    \begin{itemize}
        \item[(a)] By a Pohozaev argument, namely, testing the equation against~$x\cdot\nabla\phi$, we see that
        \begin{align*}
            \frac{2}{p+1}I_{p+1}=\pi (\phi'(1))^2,
        \end{align*}
        which, combined with \eqref{Ip} gives that,
        \begin{align*}
            I_{p+1}=\frac{p+1}{8\pi}I_p^2.
        \end{align*}

        \item[(b)] The algebraic relation between~$I_{p-1}$ and~$I_p$ is unclear, but using that~$\max\phi=\phi(0)$ and~$\phi>0$ we have that,
        \begin{align*}
            \frac{I_p}{\phi(0)}<I_{p-1}< \phi(0)^{p-1} \pi.
        \end{align*}
        By using the H\"older inequality, we get another relation:
        \begin{align*}
            I_{p-1}\leq I_p^{1-\frac{1}{p}}\pi^{\frac{1}{p}} = (-2\phi'(1))^{1-\frac{1}{p}} \pi.
        \end{align*}
    \end{itemize}

    \begin{remark}
        Observe that
        \begin{itemize}
            \item[(i)] ~$I_{p+1}$ is equivalent to the Dirichlet energy of the solution;
            \item[(ii)] $I_{p-1}$ controls both~$I_p$ and~$I_{p+1}$ up to multiplication by suitable powers of~$\phi(0)$.
        \end{itemize}
        We will see later that these hold also for general solutions.
    \end{remark}

\subsection{A rigidity result for entire solutions}

    Type I spikes are modeled on the solutions used in \cite{DY}. Concerning this point, we will need the following classification
    result about entire solutions of \eqref{eq:entire eqn}. We could not find a proof of this classification result, which is why we
    provide a sketchy proof, based however on well known ideas first pushed forward in \cite{bm} and \cite{cli2}.
    \begin{proposition}\label{prop:entire solution}
        Let~$w\in L^1_{\rm loc}(\R^2)$ be a distributional solution of
        \begin{align}\label{eq:entire eqn}
            \begin{cases}
               & -\Delta w = [w]_+^p  \mbox{ in } \R^2, \vspace{2mm} \\
            &\int\limits_{\R^2} [w]_+^p \dd{x} \leq C_0<+\infty.
            \end{cases}
        \end{align}
        Then~$w\in C^2(\R^2)$ and~$[w]_+\in L^\infty(\R^2)$.
        Moreover, denoting
        \begin{align*}
            \beta_p\equiv \beta_p(w)\coloneqq \frac{1}{2\pi}\int\limits_{\R^2} [w]_+^p\dd{x},
        \end{align*}
        we have that:
        \begin{itemize}
            \item either~$\beta_p=0$, and then $w$ is a non-positive constant,
            \item or~$\beta_p>0$, and then $w$ is, up to a translation, radial and takes the form,
                    \begin{align}\label{eq:entire solution}
                        w(x)=
                        \begin{cases}
                            \frac{1}{R_p^{\frac{2}{p-1}}}\phi\parenthesis{\frac{x}{R_p}}, & 0\leq |x|\leq R_p, \vspace{2mm} \\
                            \frac{1}{R_p^{\frac{2}{p-1}}}\phi'(1)\log\parenthesis{\frac{|x|}{R_p}}, & |x|>R_p,
                        \end{cases}
                    \end{align}
                    where~$R_p>0$ is uniquely defined in terms of $\beta_p$ as follows,
                    \begin{align}\label{Wpdef}
                        \beta_p
                        =\frac{1}{2\pi} \frac{I_p}{R_p^{\frac{2}{p-1}}}=\frac{-\phi'(1)}{R_p^{\frac{2}{p-1}}}.
                    \end{align}
        \end{itemize}
    \end{proposition}
The solution with~$R_p=1$ will be denoted as~$w^*$, namely,
    \begin{align}
        w^*(x)=\begin{cases}\label{wstar}
            \phi(x), & 0\leq |x|\leq 1, \\
            \phi'(1)\log(|x|), & |x|\geq 1.
        \end{cases}
    \end{align}

    \begin{proof}[Proof of Proposition \ref{prop:entire solution}]

    {\bf Step 1: Regularity of weak solutions}. We have the following
    \begin{lemma}\label{lemma:regularity of entire solution}
        Let~$w\in L^1_{loc}(\R^2)$ be a distributional solution of~\eqref{eq:entire eqn}.
        Then $w\in C^2(\R^2)$ and $[w]_+\in L^\infty(\R^2)$.
    \end{lemma}

    \begin{proof}[Proof of Lemma~\ref{lemma:regularity of entire solution}]
        We argue as in~\cite{bm} by a suitable decomposition of $w$. Let $x_0\in\R^2$, $r>0$
        and write $w=w_1+w_2$ in $B_{2r}(x_0)$ where,
        \begin{align*}
            \begin{cases}
                -\Delta w_1=[w]_+^p & \mbox{ in } B_{2r}(x_0), \\
                w_1=0 & \mbox{ on } \partial B_{2r}(x_0),
            \end{cases}
            \qquad \mbox{ and } \qquad
            \begin{cases}
                -\Delta w_2= 0 & \mbox{ in } B_{2r}(x_0), \\
                w_1=w & \mbox{ on } \partial B_{2r}(x_0).
            \end{cases}
        \end{align*}
        Since $[w]_+^p\in L^1(\R^2)$, then according to \cite{stam} we have that $w_1\in W_0^{1,t}(B_{2r}(x_0))$ for any $t\in [1,2)$, with
        \begin{align*}
            \|w_1\|_{W^{1,t}(B_{2r}(x_0))} \leq C(t, C_0).
        \end{align*}
        Since we are in dimension $d=2$, by the Sobolev embedding we have $w_1\in L^q(B_{2r}(x_0))$ for any $1\leq q<+\infty$ with
        \begin{align*}
            \|w_1\|_{L^q(B_{2r}(x_0))}\leq C(q, C_0).
        \end{align*}
        Moreover, by the maximum principle $w_1\geq 0$ in $B_{2r}(x_0)$. By applying the mean value theorem to the harmonic part
        $w_2=w-w_1(\leq w)$, for any $x_1\in B_r(x_0)$ we have that,
        \begin{align*}
            [w_2(x_1)]_+ \leq \fint\limits_{B_{r}(x_1)} [w]_+ \dd{x} \leq \frac{1}{\pi r^2}\parenthesis{\int\limits_{B_{r}(x_1)} [w]_+^p\dd{x}}^{\frac{1}{p}} |B_r(x_1)|^{1-\frac{1}{p}}
            \leq \frac{ C_0^{1/p}}{\pi r^{1+\frac{1}{p}}}.
        \end{align*}
        Thus~$\|[w_2]_+\|_{L^\infty(B_r(x_0))} \leq  C_0^{1/p}(\pi r^{1+\frac{1}{p}})^{-1}$.

        As a consequence,~$[w]_+\leq [w_1]_+ + [w_2]_+ \in L^q(B_r(x_0))$ for any~$q\in [1,+\infty)$, with
        \begin{align*}
            \|[w]_+\|_{L^q(B_r(x_0))} \leq C(q,C_0,r).
        \end{align*}
        Now $-\Delta w_1=[w]_+^p$ has better integrability properties.
        An iteration of the splitting argument shows that $w_1\in L^\infty(B_{r/2}(x_0))$ and so is $[w]_+$, with
        \begin{align*}
            \|[w]_+\|_{L^\infty(B_{r/2}(x_0))} \leq C(C_0,r,p).
        \end{align*}

        Noting that $t\mapsto [t]_+^p$ is a~$C^1$ function (recall $p>1$), by a bootstrap argument we conclude that
        $w\in C^{2,\gamma}(\R^2)$ for any~$\gamma\in (0,1)$.
    \end{proof}

    Consequently, all the distributional solutions are actually classical solutions.

    \

    {\bf Step 2: The case $\beta_p=0$}.
    If $\beta_p=0$, then $w$ is non-positive and harmonic, which is necessarily constant by Liouville theorem.

    \


    {\bf Step 3. The case of $\beta_p>0$}.
    We first use the argument, based on a Green representation formula, adopted in \cite{cli2} to come up with the following decay estimate:
    there exists a constant~$C>0$ such that
    \begin{align*}
        -\beta_p \ln(1+|x|) -C \leq w(x) \leq -\beta_p \ln(1+|x|) +C, \qquad \forall x\in\R^2.
    \end{align*}
    In particular, as in \cite{cli2} we have that,
    \begin{align}\label{eq:decay of w}
        \frac{w(x)}{\ln|x|} \to -\beta_p, \quad \mbox{ uniformly as } |x|\to+\infty,
    \end{align}
    and, in polar coordinates $(r,\theta)$ for $\R^2$,
    \begin{align*}
        r\frac{\partial w}{\partial w}\to -\beta_p, & & \frac{\partial w}{\partial \theta}\to 0, & & \mbox{ as } r\to+\infty.
    \end{align*}
    At this point, by using a moving plane argument as in \cite{cli2} we conclude that $w$ is radial with respect to some point and is radially decreasing.
    Thus, after a translation if necessary, we may assume that $w$ is radial with respect to the origin and that there exists a unique $R>0$ such that
    \begin{align*}
        \braces{x\in\R^2\mid w(x)>0} = B_R(0)\subset \R^2.
    \end{align*}
    By the uniqueness of the positive solutions in a ball of the Emden equation, we conclude that
    \begin{align*}
        w(x)=\frac{1}{R^{\frac{2}{p-1}}} \phi(\frac{x}{R}), \quad \mbox{ for } |x|\leq R,
    \end{align*}
    where we recall that $\phi$ is the unique solution of \eqref{eq:Emden}. In the outer domain $\R^2\setminus B_R(0)$,
    $w$ is harmonic with the decay shown in \eqref{eq:decay of w}, whence, by standard ODE theory, it
    necessarily takes the form,
    \begin{align*}
        w(x)=\frac{1}{R^{\frac{2}{p-1}}}\phi{'}(1)\ln\frac{|x|}{R},\qquad \forall |x|\geq R,
    \end{align*}
    the coefficients being uniquely defined by the $C^1$ continuity at the free boundary $\braces{w=0}$.\\
    Since all the argument adopted along the proof work exactly as they stand in \cite{cli2}, we will not provide the details here to avoid repetitions.
    \end{proof}

\bigskip

    \begin{remark}
        For the non-constant entire solution defined in \eqref{eq:entire solution}, we have seen that to each $w$ there correspond a
        unique~$R_p>0$, so that we may write,
        \begin{align*}
            \beta_p(R_p)= \frac{1}{2\pi}\int\limits_{\R^2} [w]_+^p\dd{x}
            =\frac{1}{2\pi} \frac{I_p}{R_p^{\frac{2}{p-1}}}=\frac{-\phi'(1)}{R_p^{\frac{2}{p-1}}}.
        \end{align*}
        Moreover, by an evaluation based on Pohozaev identity we have that,
        \begin{align*}
            \beta_{p+1}(R_p)\equiv \frac{1}{2\pi}\int\limits_{\R^2} [w]_+^{p+1}\dd{x}
            =\frac{(p+1)}{4}\beta_p^2.
        \end{align*}
        Another relevant quantity is,
        \begin{align*}
           \beta_{p-1}(R_p)\coloneqq \frac{1}{2\pi}\int\limits_{\R^2} [w]_+^{p-1}\dd{x}
        \end{align*}
        for which we again have no algebraic evaluation but only the direct estimate
        \begin{align*}
            \frac{\beta_p}{\sup w}<\beta_{p-1} < \frac{1}{2\pi}(\sup w)^{p-1}|\Omega|, & &
           \beta_{p-1}\leq \beta_p^{1-\frac{1}{p}} \parenthesis{\frac{|\Omega|}{2\pi}}^{\frac{1}{p}}.
        \end{align*}
        Note also that~$\beta_{p-1}, \beta_p, \beta_{p+1}$ are either simultaneously positive or vanishing.
        They have different roles in the analysis of the solutions:~$\beta_p$ describes the decaying rate at infinity as
        the above result shows, ~$\beta_{p+1}$ is equivalent to the Dirichlet energy while $\beta_{p-1}$, according to~\eqref{eq:entire solution},
        satisfies,
        \begin{align}\label{bquant}
           \beta_{p-1}(R_p)=\frac{1}{2\pi}I_{p-1},
        \end{align}
        which is a constant independent of~$R_p$.
    \end{remark}


For later purposes we remark that the equation,
    \begin{align}\label{rinv1}
        -\Delta  w= [w]_+^p
    \end{align}
    is invariant under the following transformations:
    \begin{align}\label{rinv2}
        {\mathcal{R}}_t w(x) \coloneq t^{\frac{2}{p-1}}w(tx)
    \end{align}
    for any~$t >0$.
    The masses in different scales transform in the following way:
    \begin{align}\label{rinv3}
       \beta_{p-1}({\mathcal{R}}_t w) =\beta_{p-1}(w), & &
        \beta_p ({\mathcal{R}}_t w) = t^{\frac{2}{p-1}} \beta_{p}(w), & &
        \beta_{p+1} ({\mathcal{R}}_t w) = t^{\frac{4}{p-1}} \beta_{p+1}(w).
    \end{align}
    In particular, the decaying rate $\beta_p(w)$ (see \eqref{eq:entire solution} and \eqref{Wpdef}) in Proposition~\ref{prop:entire solution} is
    not quantized and can be any positive number, unlike $\beta_{p-1}$ which always assume the value shown in \eqref{bquant}.

\bigskip

\section{Model solutions}\label{sec3}
We collect some facts about \eqref{eq:v form} and discuss the relation with the Dancer-Yan profiles.

\begin{remark}
    Note that in case~$\lambda\leq C|\alpha|$ for some~$C<+\infty$, then
    \begin{align*}
        \int_{\Omega} [v-1]_+^p=\frac{1}{|\alpha|^p}\leq \frac{C}{|\alpha|^{p-1}\lambda}=C\varepsilon^2.
    \end{align*}
\end{remark}

\

In particular, since we are in dimension $d=2$, problem \eqref{eq:v form} admits the following symmetry property: for each~$t>0$ and~$x_0\in\R^2$, the function
\begin{align*}
    \mathcal{S}_t v(y)\coloneqq t^{\frac{2}{p-1}} \parenthesis{ v(x_0+ ty) -1 }+1, \qquad \forall y\in \Omega(x_0,t)\coloneqq \frac{\Omega -x_0}{t},
\end{align*}
satisfies the equation
\begin{align*}
        -\varepsilon^2\Delta \parenthesis{\mathcal{S}_t v} = [\mathcal{S}_t v -1]_+^p \quad \mbox{ in }\quad  \Omega(x_0,t)
\end{align*}
and the corresponding integral for~$\mathcal{S}_t v$ becomes
\begin{align*}
    \int\limits_{\Omega(x_0,t)} [\mathcal{S}_t v -1]_+^{p+k} \dd y
    = t^{\frac{2(k+1)}{p-1}}\int\limits_{\Omega} [v-1]_+^{p+k} \dd x
    =\frac{t^{\frac{2(k+1)}{p-1}}}{|\alpha|^p}, \qquad k=-1,0,1.
\end{align*}
In particular, note that for~$k=-1$, the~$(p-1)$-mass is invariant under the $\mathcal{S}_t$-transform.

\subsection{The Dancer-Yan model profiles}\label{subs:3.1}
Recall that $\phi$ denotes the unique solution of the Emden equation \eqref{eq:Emden}.
Let~$R_2=\frac{1}{\sqrt{\pi}}$ and~$\Omega=B_{R_2}(0)$ so that~$|\Omega|=1$, and let~$b\leq a$.
Then the function
\begin{align*}
    U_{\varepsilon,a,b}(x)\coloneqq
    \begin{cases}
        a+\parenthesis{\frac{\varepsilon}{s_\varepsilon}}^{\frac{2}{p-1}}\phi\parenthesis{\frac{x}{s_\varepsilon}}, & 0\leq |x|\leq s_\varepsilon, \vspace{2mm} \\
        a+ (a-b) \frac{\ln(|x|/s_\varepsilon)}{\ln(\sqrt{\pi}s_\varepsilon)}, & s_\varepsilon\leq |x|\leq R_2=\frac{1}{\sqrt{\pi}}
    \end{cases}
\end{align*}
satisfies
\begin{align*}
    \begin{cases}
        -\Delta U_{\varepsilon,a,b} = \frac{1}{\varepsilon^2 }[U_{\varepsilon,a}-a]_+^p, & \mbox{ in } B_{R_2}, \vspace{2mm} \\
        \qquad \; U_{\varepsilon,a,b}=b, & \mbox{ on } \partial B_{R_2}(0).
    \end{cases}
\end{align*}
Here the~$s_\varepsilon$ is uniquely defined by imposing the~$C^1$ continuity of the solution:
\begin{align}\label{eq:C1-eps}
    \parenthesis{\frac{\varepsilon}{s_\varepsilon}}^{\frac{2}{p-1}} \phi^{'}(1) = \frac{a-b}{\ln(\sqrt{\pi}s_\varepsilon)}.
\end{align}
These solutions was first used in the context of singularly perturbed problems by Dancer and Yan in \cite{DY}.
Note that~$\lim\limits_{\varepsilon\to 0} s_\varepsilon=0$. Of course, we are interested in the case~$a=1$ and~$b=0$.
In particular,~$[U_{\varepsilon,a,b}-a]_+ \to 0$ uniformly as~$\varepsilon\to 0$, while
\begin{align*}
    \lim_{\varepsilon\to0+0}U_{\varepsilon,a,b}(x)
    =\begin{cases}
        a,& x= 0, \vspace{2mm}\\
        b,& x\neq 0.
    \end{cases}
\end{align*}

Observe that,
\begin{align*}
    \int\limits_{B_{R_2}} \frac{1}{\varepsilon^2 }[U_{\varepsilon,a,b}-a]_+^p
    =&\frac{s_\varepsilon^2}{\varepsilon^2} \int_{B_{s_\varepsilon}} \parenthesis{\frac{\varepsilon}{s_\varepsilon}}^{\frac{2p}{p-1}} \phi\parenthesis{\frac{y}{s_\varepsilon}}^p \frac{\dd{y}}{s_\varepsilon^2}
    =\parenthesis{\frac{\varepsilon}{s_\varepsilon}}^{\frac{2}{p-1}} \parenthesis{\int\limits_{B_1} \phi^p\dd{x}} \\
    =& \frac{a-b}{\phi{'}(1)\ln(\sqrt{\pi}s_\varepsilon)} I_p \to 0 \quad \mbox{ as } \varepsilon\to 0.
\end{align*}
Meanwhile for any~$t>1$,
\begin{align*}
    \int\limits_{B_{R_2}} \parenthesis{\frac{1}{\varepsilon^2 }[U_{\varepsilon,a,b}-a]_+^p}^t \dd y
    =&\frac{s_\varepsilon^2}{\varepsilon^{2t}} \int_{B_{s_\varepsilon}} \parenthesis{\frac{\varepsilon}{s_\varepsilon}}^{\frac{2p}{p-1}t} \phi\parenthesis{\frac{y}{s_\varepsilon}}^{pt} \frac{\dd{y}}{s_\varepsilon^2} \\
    =&\parenthesis{\frac{\varepsilon}{s_\varepsilon}}^{\frac{2}{p-1}t}\frac{1}{s_\varepsilon^{2(t-1)}} \parenthesis{\int\limits_{B_1} \phi^{pt}\dd x} \\
    =&\parenthesis{\frac{a-b}{\phi'(1)\ln(\sqrt{\pi}s_\varepsilon)}}^t \frac{1}{s_\varepsilon^{2(t-1)}} I_{pt}
      \to +\infty \quad \mbox{ as } \varepsilon\to 0,
\end{align*}
and
\begin{align*}
    \int\limits_{B_{R_2}} \frac{1}{\eps^2} \parenthesis{[U_{\varepsilon,a,b}-a]_+^p}^t \dd y
    =&\frac{s_\varepsilon^2}{\varepsilon^{2}} \int_{B_{s_\varepsilon}} \parenthesis{\frac{\varepsilon}{s_\varepsilon}}^{\frac{2p}{p-1}t} \phi\parenthesis{\frac{y}{s_\varepsilon}}^{pt} \frac{\dd{y}}{s_\varepsilon^2} \\
    =&\parenthesis{\frac{\varepsilon}{s_\varepsilon}}^{\frac{2pt}{p-1}-2} \int\limits_{B_1} \phi^{pt}\dd x \\
    =&\parenthesis{\frac{a-b}{\phi{'}(1)\ln(\sqrt{\pi}s_\varepsilon)}}^{p(t-1)+1}I_{pt}.
\end{align*}
Note that for~$t=\frac{p-1}{p}$, we have that
\begin{align*}
    \int\limits_{B_{R_2}} \frac{1}{\eps^2} [U_{\varepsilon,a,b}-1]_+^{p-1}\dd y
    = I_{p-1}>0
\end{align*}
which is a universal constant for any~$\varepsilon>0$.
This will be relevant for later developments.
On the other side, if~$t>\frac{p-1}{p}$, then
\begin{align*}
    \int\limits_{B_{R_2}} \frac{1}{\eps^2} \parenthesis{[U_{\varepsilon,a,b}-a]_+^p}^t \dd y
    =\parenthesis{\frac{a-b}{\phi{'}(1)\ln(\sqrt{\pi}s_\varepsilon)}}^{p(t-\frac{p-1}{p})}I_{pt}\to 0, \quad \mbox{ as } \varepsilon\to 0.
\end{align*}

\

Concerning the integral constraint,
\begin{align}\label{eq:integral-eps}
    \frac{1}{|\alpha|^p} = \int\limits_{B_{R_2}} [U_{\varepsilon,a,b} -a ]_+^p
    =&\varepsilon^2\parenthesis{\frac{\varepsilon}{s_\varepsilon}}^{\frac{2}{p-1}} I_p \\ \nonumber
    =&\parenthesis{\frac{\varepsilon}{s_\varepsilon}}^{\frac{2}{p-1}p} s_\varepsilon^2  I_p \\  \nonumber
    =&\parenthesis{\frac{a-b}{\phi{'}(1)\ln(\sqrt{\pi}s_\varepsilon)}}^p s_\varepsilon^2 I_p,
\end{align}
which uniquely defines as well the value of~$\varepsilon$ and hence also of $s_\varepsilon$.

Combining~\eqref{eq:C1-eps} and~\eqref{eq:integral-eps}, we see that,
\begin{align*}
    \parenthesis{\frac{\varepsilon}{s_{\varepsilon}}}^{\frac{2}{p-1}}=\frac{a-b}{\phi{'}(1)\ln(\sqrt{\pi}s_\varepsilon)}
    =\parenthesis{\frac{1}{|\alpha|^p I_p s_\varepsilon^2}}^{\frac{1}{p}},
\end{align*}
and consequently,
\begin{align*}
    s_\varepsilon= \varepsilon^p |\alpha|^{\frac{p(p-1)}{2}} I_p^{\frac{p-1}{2}}
    =\lambda^{-\frac{p}{2}} I_p^{p-1}.
\end{align*}

In this model case, to recover a Dancer-Yan spike at the origin, we should make the following rescaling:
as~$\varepsilon_n\to 0$ and~$\lambda_n\to +\infty$, taking~$x_n=0$, and~$v_n= U_{\varepsilon_n, 1,0}$ as a sequence of solutions in~$B_{R_2}(0)$, then
\begin{align}\label{eq:rescale-2d model}
    \widetilde{v}_n(y)\coloneqq {\phi{'}(1)\ln(\sqrt{\pi}s_{\varepsilon_n})}\parenthesis{v_n(x_n+ s_{\varepsilon_n}y) -1 },
\end{align}
are defined in the domains
\begin{align*}
    \Omega_n \coloneqq \frac{\Omega -x_n}{s_{\varepsilon_n}}= B_{\sqrt{\pi}/s_{\varepsilon_n}} \nearrow \R^2.
\end{align*}
In particular $\widetilde{v}_n\to  w^*$ locally uniformly, where $w^*$ is defined in \eqref{wstar}.

\bigskip

\subsection{1D solutions}\label{subs:3.3}
In this section we construct solutions of the equation in \eqref{eq:entire eqn} which do not satisfy the integral
bound, that is, a class of ``infinite mass'' solutions of that equation. For any fixed $a>0$, let~$u(t)$ be the unique solution of
\begin{align}\label{eq:1D eqn}
    \begin{cases}
        -u''(t)= [u(t)]_+^p, \qquad \mbox{ for } t\geq 0, \vspace{2mm}\\
        u(0)=a, \quad u'(0)=0.
    \end{cases}
\end{align}
Thus $u(t)$ is non-increasing and consequently there exists a unique~$t_0\in(0,+\infty]$ such that~$u(t)>0$ for~$t\in [0,t_0)$. Therefore we have,
\begin{align*}
    -u''=u^p, \qquad \mbox{ in } (0,t_0).
\end{align*}

Multiplying both sides of the equation by $u'$ and integrating from~$0$ to~$t(<t_0)$, we have that,
\begin{align*}
    u'(t)^2 = \frac{2}{p+1}(a^{p+1}-u(t)^{p+1}).
\end{align*}
Since~$u'<0$, we also have that
\begin{align*}
    \frac{\dd u}{\sqrt{a^{p+1}- u^{p+1}}}= -\sqrt{\frac{2}{p+1}}\dd t,
\end{align*}
that is, $u(t)$ is implicitly defined for $t\in [0,t_0]$ as follows,
\begin{align*}
    \int_u^a \frac{\dd s}{\sqrt{a^{p+1}-s^{p+1}}} = \sqrt{\frac{2}{p+1}} \cdot t \, .
\end{align*}
In particular,~$t_0$ satisfies
\begin{align*}
    \int_0^a \frac{\dd s}{\sqrt{a^{p+1}-s^{p+1}}} = \sqrt{\frac{2}{p+1}} \cdot t_0
\end{align*}
and~$t_0$ is finite since~$p>1$.
For~$t\geq t_0$,~$u(t)$ is linear and since~$u'(t_0)=-\sqrt{\frac{2}{p+1}}a^{\frac{p+1}{2}}$, we see that
\begin{align*}
    u(t)= -\sqrt{\frac{2}{p+1}}a^{\frac{p+1}{2}}(t-t_0), \qquad \forall t\geq t_0.
\end{align*}

Next we extend $u$ to the whole real line by even reflection: $u(-t)\coloneqq u(t)$ for any~$t\geq 0$. This extended~$u$ is~$C^2$ and solves~\eqref{eq:1D eqn}.

Finally, set~$w^{(1D)}(x^1, x^2)\coloneqq u(x^1)$.
Then~$w^{(1D)}$ solves~\eqref{eq:entire eqn} and has infinite~$(p+k)$-mass, for~$k=-1,0,1$ and
the plasma region of such solution is unbounded.
We rule out these solutions by assuming \eqref{Hypo on p-1} in Theorem \ref{thm:ALT near maximum}.\\
Actually, for any fixed $\alpha\in [0,2\pi]$,
\begin{align*}
    w^{(1D)}_{\alpha}(x^1,x^2)\coloneqq u(x^1\cos\alpha+x^2\sin\alpha),
\end{align*}
is a family of one-dimensional infinite mass solutions.

\bigskip

\section{Blow up analysis along interior maximum}\label{sect:blowup max}
First of all we prove that, unlike the case $d\geq 3$ (\cite{BJW}), the natural integral bound
\eqref{Hypo on p-1} implies that $[v_n-1]_+ \to 0$ locally uniformly in~$\Omega$.
    \begin{proposition}\label{prop:local vanishing}
        Let~$v_n$ be a sequence of solutions of~\eqref{eq:blow up eqn} satisfying~\eqref{Hypo on p-1}.
        Then~$[v_n-1]_+ \to 0$ locally uniformly in~$\Omega$.
    \end{proposition}

    \begin{proof}
        Argue by contradiction and assume that there exists a sequence~$x_n\in \Omega_0\Subset \Omega$ such that~$v_n(x_n)\geq 1+\sigma$ for some~$\sigma >0$.
        Along a subsequence we may assume that $x_n\to x_0 \in \overline{\Omega_0}\Subset \Omega$.
        For each~$n\geq 1$, set
        \begin{align*}
            \hat{v}_n(y)\coloneqq v_n(x_n + \eps_n y), \qquad
            \forall y\in \widehat{\Omega}_n\coloneqq \frac{\Omega-x_0}{\eps_n},
        \end{align*}
        which satisfies
        \begin{align*}
            \begin{cases}
                -\Delta \hat{v}_n=[\hat{v}_n-1]_+^p,  & \mbox{ in } \widehat{\Omega}_n, \vspace{2mm} \\
                \hat{v}_n\geq 0, & \mbox{ in } \widehat{\Omega}_n.
            \end{cases}
        \end{align*}
        According to \eqref{Hypo on p-1}, we have that,
        \begin{align}\label{Lp-1}
            \int\limits_{\widehat{\Omega}_n} [\hat{v}_n-1]_+^{p-1}\dd{y}
            =\frac{1}{\eps_n^2}\int\limits_{\Omega} [v_n-1]_+^{p-1}\dd{x} \leq H_{p-1}.
        \end{align}
        By \cite[Theorem 6.1]{BJW} we have that, possibly along a subsequence, $\hat{v}_n\to \hat{v}$ in~$C^2_{\rm loc}(\R^2)$, for some~$\hat{v}\in C^2(\R^2)$ which solves,
        \begin{align*}
            \begin{cases}
                -\Delta\hat{v}=[\hat{v}-1]_+^p, & \mbox{ in } \R^2, \vspace{2mm} \\
                \hat{v}\geq 0, & \mbox{ in } \R^2.
            \end{cases}
        \end{align*}
        By a well known argument (see \cite{MitP}) we see that $\hat{v}$ is constant whence, in view of \eqref{Lp-1}, we have that $\Delta\hat{v}\equiv 0$.
        On the other side by assumption we have $\hat{v}_n(0)=v_n(x_n)\geq 1+\sigma>1$, so that
        $\hat{v}(0)\geq 1+\sigma$ and $-\Delta\hat{v}(0)=[\hat{v}(0)-1]_+^p>0$, which is contradiction.
    \end{proof}

\bigskip

    As a consequence, for any~$\Omega_0\Subset \Omega$, we have
    \begin{align*}
        \limsup_{n\to+\infty}\parenthesis{\sup_{\Omega_0} v_n } \leq 1,
    \end{align*}
    and we assume w.l.o.g that,
    \begin{align*}
        \sup_{\Omega_0} v_n \leq 2, \qquad \forall n\geq 1.
    \end{align*}
    Thus \eqref{Hypo on p-1} implies that,
    \begin{align*}
        \int\limits_{\Omega_0} \frac{1}{\eps_n^2} [v_n-1]_+^p \dd{x}
        \leq \int\limits_{\Omega_0}\frac{1}{\eps_n^2}[v_n-1]_+^{p-1}\dd{x} \leq H_{p-1}.
    \end{align*}

    \begin{definition}[Regular points]\label{def:reg} A point $x\in \om$ is said to be \underline{regular} w.r.t. a sequence of solutions $v_n$ of \eqref{eq:blow up eqn}
    if there exists $r>0$ such that
    \begin{align*}
        v_n|_{B_r(y)}\le 1, \qquad \forall y\in B_r(x), \quad \forall n\geq 1.
    \end{align*}
    \end{definition}

 \begin{remark}\label{rem:reg} If $x$ is regular w.r.t. to $v_n$ then the sequence~$\{v_n|_{B_r(x)}\}_{n\geq 1}$ is a sequence of bounded
 (from below by~$0$ and from above by~$1$) harmonic functions. Thus it sub-converges in~$C^k(B_{r/2}(x))$
 to a harmonic function which takes values in~$[0,1]$ as well.
 \end{remark}

\bigskip

As mentioned in the introduction, the subtle point is to deal with those interior points around which $[v_n-1]_+$ is positive along a subsequence.
Therefore we assume w.l.o.g. that there exist~$x_n\in\Omega$, $x_n\to x^*\in \Omega$ such that
    \begin{align*}
        v_n(x_n)=\max_{\Omega} v_n>1.
    \end{align*}
    From Proposition \ref{prop:local vanishing} we know that~$v_n(x_n)\to 1$ and we study these vanishing spikes by
    using a refined rescaling. Motivated by the Dancer-Yan solutions (see subsection \ref{subs:3.1}),
    we introduce a new parameter~$s_n>0$ defined as in \eqref{thetan1},
    \begin{align*}
        \parenthesis{\frac{\eps_n}{s_n}}^{\frac{2}{p-1}} \phi'(1)\ln(\sqrt{\pi} s_n)=1.
    \end{align*}
    Remark that for~$\eps_n$ sufficiently small, $s_n$ is uniquely defined and~$s_n\to 0^+$, as $\eps_n\to 0^+$.
    We assume w.l.o.g. that $s_n\leq \frac{1}{\sqrt{\pi}}$. For later convenience, we also denote
    \begin{align*}
        \theta_n\coloneqq \phi'(1)\ln(\sqrt{\pi}s_n),
    \end{align*}
    which, since~$\phi'(1)<0$, satisfies $\lim\limits_{n\to+\infty}\theta_n=+\infty$.
    Note that
    \begin{align}\label{2ways}
        \eps_n^2 = \frac{s_n^2}{\theta_n^{p-1}} \quad \mbox{ or either }\quad \frac{\eps_n^2}{\theta_n} = \frac{s_n^2}{\theta_n^{p}}.
    \end{align}

    We will analyse the rescaled functions,
    \begin{align}\label{eq:refined rescaling2}
        \tilde{v}_n(y)\coloneqq \phi'(1)\ln(\sqrt{\pi}s_n) \parenthesis{ v_n(x_n+ s_n y)-1}
        =\theta_n\parenthesis{ v_n(x_n+ s_n y)-1}
    \end{align}
    which satisfy
    \begin{align*}
        -\Delta \tilde{v}_n=[\tilde{v}_n]_+^p, \quad \mbox{ in }\quad \widetilde{\Omega}_n\coloneqq\frac{\Omega-x_n}{s_n}.
    \end{align*}
    Note that the~$\tilde{v}_n$'s are sign changing function in general. However, we have the following integral bounds:
    \begin{align*}
        \int\limits_{\widetilde{\Omega}_n}[\tilde{v}_n(y)]_+^{p-1}\dd{y}
        =\theta_n^{p-1}s_n^{-2}\int\limits_{\Omega}[v_n(x)-1]_+^{p-1}\dd{x}
        =\frac{1}{\eps_n^2}\int\limits_{\Omega}[v_n(x)-1]_+^{p-1}\dd{x}
        \leq H_{p-1}
    \end{align*}
    due to the Hypothesis~\eqref{Hypo on p-1}.
    Moreover, for any~$\Omega_0\Subset \Omega$, we have that~$\sup\limits_{\Omega_0}v_n \in [0,2]$ for any $n$ large, and then, by setting,
    \begin{align*}
        \widetilde{\Omega}_0\coloneqq\frac{\Omega_0 -x_n}{s_n},
    \end{align*}
    we see that (recall \eqref{2ways}),
    \begin{align*}
        \int\limits_{\widetilde{\Omega}_0} [\tilde{v}_n(y)]_+^p\dd{y}
        = \theta_n^p s_n^{-2} \int\limits_{\Omega_0}[v_n(x)-1]_+^p\dd{x}
        =\frac{\theta_n}{\eps_n^2} \int\limits_{\Omega_0} [v_n(x)-1]_+^p\dd{x}\leq H_p
    \end{align*}
    where we used \eqref{Hypo on p}.

      \begin{remark}\label{rem:spikeresc}
It is evident from \eqref{eq:refined rescaling2} that we are scaling the domain variable while blowing up the function with asymptotically
different rates. This is a characteristic feature of the $d=2$ refined rescaling (\cite{DY}) which allows one to catch these ``microscopic'' spikes, in striking contrast with the case $d\geq 3$ (see \cite{FW,W} and \cite{BJW}).
    \end{remark}

We are ready to prove Theorem \ref{thm:ALT near maximum}.

    \begin{proof}[Proof of Theorem \ref{thm:ALT near maximum}]
        If all interior points are regular according to Definition \ref{def:reg}, then in view of Remark \ref{rem:reg} we are in case (A).
        The estimate \eqref{eq:Vanishing-interior est} is then just a consequence of the mean value property of harmonic functions.\\
Therefore we assume without loss of generality that there exists $x_n\in \Omega_0\Subset \Omega$ such that
$$
x_n\to x^*\in \Omega_0
$$
and
        \begin{align}\label{4.10}
            v(x_n)=\sup_{\Omega_0} v_n >1,
        \end{align}
        implying that we are in case (B).\\
        Fix~$\delta <\frac{1}{4}\dist(x^*,\partial \Omega)$, then~$\overline{B_{4\delta}(x^*)}\subset \Omega$
        as well as~$B_\delta(x^*)\subset \overline{B_{2\delta}(x_n)}\subset\Omega$ for~$n$ large which we assume w.l.o.g.  for every~$n\geq 1$.

     At this point we consider the sequences~$t_n\equiv \left(\phi(0)/\theta_n(v_n(x_n)-1)\right)^{\frac{p-1}{2}}$ and~$u_n$ defined in~\eqref{eq:un def}.
        The sequence $u_n$ satisfies \eqref{eq:un eqn} together with the integral bounds,
        \begin{align*}
            \int\limits_{B_{\frac{2\delta}{s_n t_n}} (0)} [u_n]_+^{p-1}\dd{z}
            =\int\limits_{B_{2\delta}(x_n)} \frac{1}{\eps_n^p}[v_n-1]_+^{p-1}\dd{x}  \leq H_{p-1},
        \end{align*}
        and
        \begin{align}\nonumber
            &\int\limits_{B_{\frac{2\delta}{s_n t_n}}(0)} [u_n]_+^{p} \dd{z}
            = t_n^{\frac{2}{p-1}}\int\limits_{B_{2\delta}(x_n)} \frac{\theta_n}{\eps_n^2} [v_n -1 ]_+^{p}\dd{x} \leq\\
      \label{intboundp}      &t_n^{\frac{2}{p-1}}\theta_n[v_n(x_n)-1]_+ \cdot \int\limits_{B_{2\delta}(x_n)} \frac{1}{\eps_n^2}[v_n-1]_+^{p-1}\dd{x}
            \leq \phi(0)H_{p-1}.
        \end{align}

        If~$s_n t_n \to 0^+$, then~$\frac{2\delta}{s_n t_n}\nearrow +\infty$ and the disks~$B_{\frac{2\delta}{s_n t_n}}(0)$ exhaust the plane in the limit.
        In this case, the functions~$u_n$ subconverge locally in $\R^2$, due to the following Lemma.

        \begin{lemma}\label{lemma:ALT by Harnack}
            Let~$f_n$ be a sequence of solutions of
            \begin{align*}
                \begin{cases}
                     &-\Delta f_n = [f_n]_+^p, \mbox{ in } \omega, \vspace{2mm} \\
                    &\int\limits_{\omega} [f_n]_+^p \dd{x}\leq C_1 <+\infty,
                \end{cases}
            \end{align*}
            where~$\omega\subset \R^2$ is open and bounded with smooth boundary.
            Then
            \begin{itemize}
                \item[either] {\rm (i)} there exists a subsequence~$f_{n_k}$ which converges in~$C^2_{loc}(\omega)$,

                \item[or $\quad$] {\rm (ii)} $f_n\to -\infty$ locally uniformly in~$\omega$.
            \end{itemize}
        \end{lemma}
         \begin{proof}[Proof of Lemma~\ref{lemma:ALT by Harnack}]
        Consider the decomposition~$f_n= f_{n,1}+ f_{n,2}$, where
        \begin{align*}
            \begin{cases}
                -\Delta f_{n,1}= [f_n]_+^p & \mbox{ in  } \omega, \vspace{2mm} \\
                f_{n,1}=0, & \mbox{ on } \partial \omega,
            \end{cases}
            & &
            \begin{cases}
                -\Delta f_{n,2}=0 ,& \mbox{ in  } \omega, \vspace{2mm} \\
                f_{n,2}=f_n, & \mbox{ on } \partial \omega.
            \end{cases}
        \end{align*}
      By the maximum principle $f_{n,1}$ is nonnegative and, due to the classical estimates in \cite{stam}, since
      $[f_n]_+^p$ is uniformly bounded in~$L^1$, we have that,
        \begin{align*}
            \|f_{n,1}\|_{W^{1,t}(\omega)} \leq C(t), \qquad \forall t\in(1,2).
        \end{align*}
         Since we are in dimension two, by Sobolev embedding we see that for any~$q\geq 1$,
        \begin{align*}
            \|f_{n,1}\|_{L^q(\omega)} \leq C(q).
        \end{align*}

        Observe that $f_{n,2}$ is harmonic and bounded from above, $f_{n,2}=f_n- f_{n,1}\leq f_n \leq [f_n]_+$.
        Thus for any~$x\in \omega'\Subset \omega$, the mean value property implies that,
        \begin{align*}
            f_{n,2}(x)=\fint\limits_{B_R(x)} f_{n,2}\dd{y}
            \leq\fint\limits_{B_R(x)} [f_n]_+\dd{y}
            \leq \parenthesis{\fint\limits [f_n]_+^p\dd{y}}^{\frac{1}{p}}
            \leq \parenthesis{\frac{C_1}{\pi R^2}}^{\frac{1}{p}},
            \qquad \forall B_R(x)\subset \omega.
        \end{align*}
        We deduce that~$[f_{n,2}]_+$ is bounded in~$L^\infty_{\rm loc}(\omega)$.
        Moreover, for~$\omega'\Subset \omega$ as above, let~$C(\omega,\omega')-1$ be a uniform upper bound for~$[f_{n,2}]_+$ in~$\omega'$, which is also the uniform upper bound for~$f_{n,2}$ in~$\omega'$.
        Thus~$(C(\omega,\omega')-f_{n,2})$ is a sequence a strictly positive  (bounded from below by~$1$) harmonic functions in~$\omega'$.
        The Harnack inequality implies that either there is a uniformly bounded subsequence,
        or they are uniformly divergent~\footnote{This can be seen in the following way. Consider
        the sequence~$a_n\coloneqq \inf\limits_{k\geq n} \parenthesis{\inf\limits_{x\in \omega'}
        (C(\omega,\omega')-f_k(x))}$. If~$(a_n)$ is bounded from above, then we get a convergent subsequence~$(a_{n_k})$,
        which corresponds to a bounded subsequence of harmonic functions;
        then we can get a convergent sub-subsequence. If not, then~$(a_n)$ diverges to~$+\infty$,
        so is the sequence~$C(\omega,\omega')-f_n$, hence~$f_{n}$ diverges to~$-\infty$ uniformly on~$\omega'$.} on~$\overline{\omega'}$.

        Note that~$[f_n]_+ \leq [f_{n,1}]_+ + [f_{n,2}]_+$ where the right hand side is now in~$L^q_{\rm loc}(\omega)$ for any~$q>1$. Another bootstrap argument then implies~$f_{n,1}$ is also bounded in~$L^\infty_{\rm loc}(\omega)$.

        \

        Therefore we conclude that
        \begin{itemize}
            \item[either] {\rm (i)} there exists a subsequence~$f_{n_k}$ which is bounded in~$L^\infty_{\rm loc}(\omega)$: in this case we can use bootstrap argument to conclude that a subsequence~$f_{n_{k_j}}$ actually converges in~$C^2_{\rm loc}(\omega)$;
            \item[or $\quad$] {\rm (ii)} $f_{n}\to-\infty$ locally uniformly in~$\omega$,
        \end{itemize}
as claimed.
    \end{proof}

       \medskip

        We can apply Lemma~\ref{lemma:ALT by Harnack} to $u_n$ in any $B_R(0)$ for any $R>0$,
        and noting that~$u_n$ cannot diverge to~$-\infty$ locally uniformly since~$u_n(0)=\phi(0)>0$, we conclude that
        there is a subsequence which converges in~$C^2_{\rm loc} (\R^2)$ to a limit function~$u_\infty\in C^2_{\rm loc}(\R^2)$. In view of \eqref{intboundp},
        we have that $u_\infty$ is an entire solution of~\eqref{eq:entire eqn} with~$u_\infty(0)=\phi(0)$, hence, in view of Proposition
        \ref{prop:entire solution}, we deduce that $u_\infty= w^*$ in~$\R^2$.
        Furthermore, by using a diagonal argument, we obtain a sequence~$R_n\nearrow +\infty$ such that~$ R_n s_n t_n < 2\delta$ and~$\|u_n - w^*\|_{C^2(B_{2R_n})}\to 0$.
        Thus we have
        \begin{align*}
            \int\limits_{B_{2R_n s_n t_n}(x_n)} \frac{1}{\eps_n^p}[v_n-1]_+^{p-1}\dd{x}
            =\int\limits_{B_{2R_n} (0)} [u_n]_+^{p-1}\dd{z} = I_{p-1} + o(1),
        \end{align*}
        and
        \begin{align*}
            \int\limits_{B_{2R_n s_n t_n}(x_n)} \frac{\theta_n}{\eps_n^2}[v_n-1]_+^{p}\dd{x}
            =\frac{1}{t_n^{\frac{2}{p-1}}}\int\limits_{B_{2R_n} (0)} [u_n]_+^{p-1}\dd{z}
            =\frac{I_p+ o(1)}{t_n^{\frac{2}{p-1}}},
        \end{align*}
        as~$n\to+\infty$.
        This corresponds to the Type I spike if $t_n\to t_\infty<+\infty$, to the Type II spike if $t_n\to+\infty$ but~$s_n t_n \to 0^+$, and
        we see that a Type II spike is characterized by the fact that the~$(p)$-mass of the sequence~$\tilde{v}_n$ vanishes.
        Note that, again in view of \eqref{intboundp}, as a byproduct of this argument we see that~$t_n$ has a positive lower bound,
        say $t_n\geq \parenthesis{\frac{I_p}{2H_p}}^{\frac{p-1}{2}}>0$.

       \bigskip

        Next we consider the case where $s_n t_n$ has a positive lower bound, which means that,
        up to subsequence, either~$s_n t_n \to \frac{1}{r_0}>0$ or~$s_n t_n \to +\infty$.
        In this case, the rescaled disks~$B_{2\delta/s_n t_n}(0)$ are uniformly bounded sets and we don't have a planar problem to attach to the limiting
        function. Instead, we can consider the functions,
        \begin{align*}
                \frac{v_n(x)-1}{\eps_n^{\frac{2}{p-1}}}
            \end{align*}
            which satisfy
            \begin{align*}
                -\Delta \parenthesis{\frac{v_n(x)-1}{\eps_n^{\frac{2}{p-1}}}}
                =\left[\frac{v_n(x)-1}{\eps_n^{\frac{2}{p-1}}}\right]_+^p \quad \mbox{ in }\quad B_{2\delta}(x_n),
            \end{align*}
            whose maximum values are,
            \begin{align*}
                \frac{v_n(x_n)-1}{\eps_n^{\frac{2}{p-1}}}
                =\frac{\theta_n}{s_n^{\frac{2}{p-1}}} \frac{\tilde{v}_n(0)}{\theta_n}
                =\frac{\phi(0)}{(s_n t_n)^{\frac{2}{p-1}} }.
            \end{align*}

            Therefore by Lemma \ref{lemma:ALT by Harnack} we have that,
            \begin{align*}
             v_n(x)-1= \frac{\eps_n^{\frac{2}{p-1}} }{(s_n t_n)^{\frac{2}{p-1}}}(f(x)+o(1)),\quad \forall\, x \in B_{\delta}(x^*),
            \end{align*}
            for some bounded function $f$ satisfying $f(x)\leq \phi(0)$.
        This corresponds to the Fading spike in (B-iii). Note that in this case the~$(p-1)$-mass satisfies,
        \begin{align*}
            \int\limits_{B_{\delta}(x^*)} \frac{1}{\eps_n^2} [v_n-1]_+^{p-1} \dd{x}
            \leq \pi \dt^2\frac{\phi^{p-1}(0) }{(s_n t_n)^2}(1+o(1)),
        \end{align*}
        which of course need not converge to some integer multiple of~$I_{p-1}$. Actually it would converge to zero whenever $s_n t_n\to+\infty$.
        However, the~$(p)$-mass will vanish in the limit as from \eqref{2ways} we have that,
          \begin{align*}
            &\int\limits_{B_{\delta}(x^*)} \frac{\theta_n}{\eps_n^2} [v_n-1]_+^{p} \dd{x}
            \leq \pi \dt^2\frac{\theta_n}{\eps_n^2}\frac{\phi^{p}(0)}{(s_n t_n)^{\frac{2p}{p-1}}}{\eps_n^{\frac{2p}{p-1}}}\leq\\
            &\pi \dt^2\frac{\phi^{p}(0)}{(s_n t_n)^{\frac{2p}{p-1}}}\theta_n\eps_n^{\frac{2}{p-1}}=
            \pi \dt^2 \frac{\phi^{p}(0)}{(s_n t_n)^{\frac{2p}{p-1}}} s_n^{\frac{2}{p-1}}(1+o(1)).
        \end{align*}
    \end{proof}

    The local spikes arising from Theorem \ref{thm:ALT near maximum} are shown in Figuer~\ref{fig:1}.

    \begin{figure}
  \centering
  \includegraphics[width=0.8\textwidth]{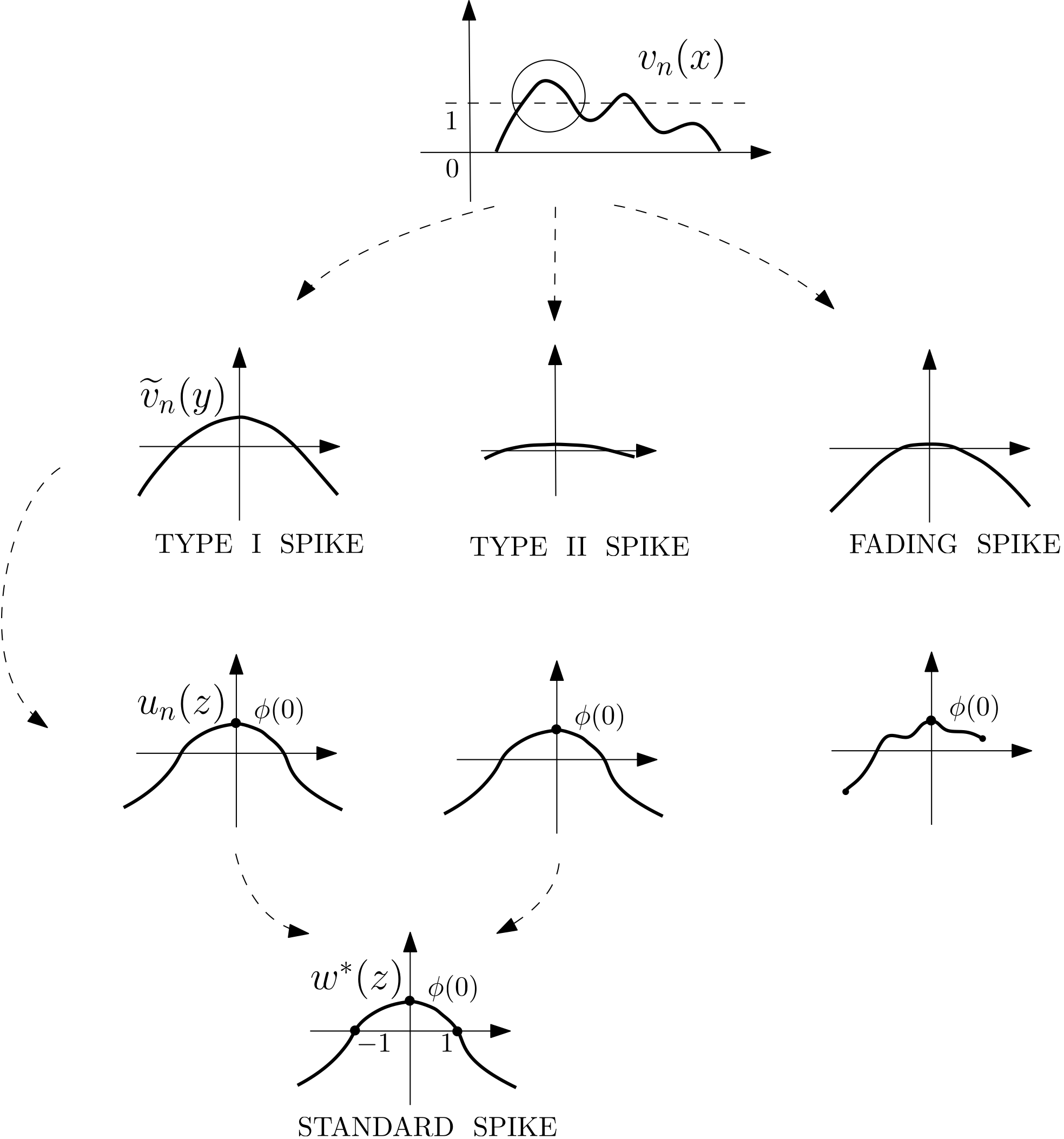}
  \caption{Spikes arising from Theorem \ref{thm:ALT near maximum}} \label{fig:1}
\end{figure}

%
%

    \begin{remark}\label{rem:4.6}
        In terms of the sequence $v_n$ satisfying \eqref{eq:blow up eqn} and the interior maximum points $x_n$ (which converge to
        $x^*\in \Omega_0\Subset \Omega$) we have the following local picture corresponding to the above alternatives: for all~$x\in B_\delta(z)$ and up to a subsequence:
        \begin{itemize}
            \item[either] {\rm (A) [Vanishing] } $v_n(x)\leq 1$, $\forall n\geq 1$,

            \item[or $\quad$] {\rm (B-i or ii) [Type I or II spikes]}
            $v_n(x)= 1+\frac{1}{t_n^{\frac{2}{p-1}}\theta_n} w^*(\frac{x-x_n}{s_n t_n})+ o(\frac{1}{\theta_n t_{n}^{\frac{2}{p-1}}})$,
            with $$|x-x_n|\leq R_ns_n t_n,\quad  t_n\to t_\infty\leq +\infty \mbox{ as }n\to+\infty.$$

            \hspace{-4mm} Moreover, for any~$R>1$,
            \begin{align*}
                \frac{1}{\eps_n^2}\int\limits_{B_{R s_n t_n}(x_n)} [v_n(x)-1]_+^{p-1}\dd{x}=I_{p-1} + o(1), & &
                \frac{\theta_n}{\eps_n^2}\int\limits_{B_{R s_n t_n}(x_n)} [v_n(x)-1]_+^{p}\dd{x}= \frac{I_p}{t_n^{\frac{2}{p-1}}} + o(1),
            \end{align*}
            as~$n\to+\infty$,
            \item[or $\quad$] {\rm (B-iii) [Fading spikes]} $v_n(x)= 1+ \phi(0)\eps_n^{\frac{2}{p-1}} (s_n t_n)^{-\frac{2}{p-1}} + o\parenthesis{\eps_n^{\frac{2}{p-1}}}$ with~$ (s_n t_n)^{-1}\to \rho\geq 0$.
            In particular, for any~$r<\delta$,
            \begin{align*}
                \frac{1}{\eps_n^2}\int_{B_r(x_n)} [v_n(x)-1]_+^{p-1} \dd{x}= \phi(0)^{p-1}\pi r^2 ((s_n t_n)^{-2} +o(1)),
            \end{align*}
            \begin{align*}
                \frac{\theta_n}{\eps_n^2}\int_{B_r(x_n)} [v_n(x)-1]_+^p\dd{x}=\phi(0)^{p}\pi r^2 s_n^{\frac{2}{p-1}}((s_n t_n)^{-\frac{2p}{p-1}}+o(1)),
            \end{align*}
            as~$n\to+\infty$.
        \end{itemize}
    \end{remark}

   \bigskip

\section{Extraction of a second spike sequence}\label{sec5}
We prove various partial results which at last will be used to prove Theorem \ref{thm:ALT many maximum}. We split the discussion into three
subsections, whose titles are meant to clarify which is the aim therein.\\

    We keep the notations in \eqref{4.10} and let~$x_{n,1}\equiv x_n\to x^*\in \Omega$ be the interior maxima of~$v_n$ such that
    $$\max\limits_{\Omega_0} v_n = v_n(x_{n,1}).$$

    If $x^*$ is regular, whence $v_n(x_n)-1\leq 0$, or either if this is a Fading sequence,
    i.e.~$v_n(x_n)-1 \leq \phi(0)\eps_n^{\frac{2}{p-1}}(A_n+o(1))$,
    then we see that there are no interior spikes at all in any~$\Omega_0\Subset\Omega$.\\
    Therefore, in this section we consider the case where
    \begin{equation}\label{Htype}
    v_n(x_{n,1}) \mbox{ yields a spike at } x^*\in \Omega_0 \mbox{ either of Type I or of Type II}.
    \end{equation}
    We wonder whether or not there is another sequence of points yielding another spike at the same point~$x^*$,
    which makes $x^*$ a multiple-spike point and, whenever this were the case, how to describe the second spike.

    \

    Let $R_{n,1}=R_n$ and $t_{n,1}=t_n$ be the quantities defined in Theorem~\ref{thm:ALT near maximum} and
    let $$x_{n,2}\in \Omega\setminus B_{2R_{n,1} s_n t_{n,1}}(x_{n,1})$$ be defined as follows,
    \begin{align*}
     v_n(x_{n,2})=\sup_{\Omega\setminus B_{2R_{n,1} s_n t_{n,1}}(x_{n,1})} v_n,
    \end{align*}
    where we assume w.l.o.g. that, possibly along a subsequence,
    \begin{align}\label{Hypxn2}
       x_{n,2}\to x^* \quad \mbox{ and }\quad   v_n(x_{n,2})-1>0.
    \end{align}
    Note that~$|x_{n,2}-x_{n,1}|\geq 2 R_{n,1}s_n t_{n,1}$ by definition while $0<v_n(x_{n,2})-1\to 0^+$ by Proposition~\ref{prop:local vanishing}.

    Let us define $t_{n,2}\coloneqq \parenthesis{\phi(0)/ \theta_n (v_n(x_{n,2}-1))}^{\frac{p-1}{2}}(\geq t_{n,1}\geq T_0)$ and
    rescale~$v_n$ as follows,
    \begin{align}\label{eq:un2 def}
        u_{n,2}(z)\coloneqq t_{n,2}^{\frac{2}{p-1}} \theta_n \parenthesis{ v_n(x_{n,2}+ s_{n} t_{n,2} z)-1 }, \qquad z\in B_{{2\delta} / {s_n t_{n,2}}}(0).
    \end{align}
    Then~$u_{n,2}(0)=\phi(0)$, and~$u_{n,2}$ satisfies,
    \begin{align*}
        -\Delta u_{n,2} = [u_{n,2}]_+^p, \qquad \mbox{ in }  B_{{2\delta} / {s_n t_{n,2}}}(0),
    \end{align*}
    as well as the integral bounds,
    \begin{align*}
        \int\limits_{B_{2\delta/ s_n t_{n,2}}(0)} [u_{n,2}(z)]_+^{p-1}\dd{z}
        = \int\limits_{B_{2\delta}(x_{n,2})}  \frac{1}{\eps_n^2}[v_{n}-1]_+^{p-1} \dd{x}\leq H_{p-1},
    \end{align*}
    \begin{align*}
        \int\limits_{B_{2\delta/s_n t_{n,2}} (0)} [u_{n,2}(z)]_+^p \dd{z}
        =t_{n,2}^{\frac{2}{p-1}}\int\limits_{B_{2\delta}(x_{n,2})} \frac{\theta_n}{\eps_n^2}[v_n(x)-1]_+^p\dd{x}
        \leq \phi(0)H_{p-1}.
    \end{align*}

    \begin{remark}\label{rem:typeIvsII}
    The fact that~$s_n t_{n,2}\geq s_n t_{n,1}$ implies that, if~$v_n(x_{n,1})$ yields a spike of Type I, then~$v_n(x_{n,2})$ could be either a
    of Type I or of Type II or Fading. On the other side, if~$v_n(x_{n,1})$ is of Type II,
    then~$v_n(x_{n,2})$ cannot be of Type I.
\end{remark}

\subsection{Two local maximizers cannot be too close each other.}\label{subs:5.1}

    We are forced to compare the relative distance between~$x_{n,2}$ and~$x_{n,1}$ with the rescaling rate $s_n t_{n,2}$.
    The following Lemma says that $|x_{n,1}-x_{n,2}|$ is much larger than $s_n t_{n,2}$.
    \begin{lemma}\label{lemma:rel dist}
       Let $x_{n,2}$, $x_{n,1}$ and $s_n t_{n,2}$ be defined as above, then we have,
        \begin{align}\label{eq:far I}
            \frac{|x_{n,1}-x_{n,2}|}{s_n t_{n,2}}\to+\infty.
        \end{align}
    \end{lemma}

    \begin{proof}
        By contradiction assume that there exists~$C>0$ such that,
        \begin{align*}
            \frac{R_{n,1} t_{n,1}}{t_{n,2}}\leq \frac{|x_{n,1}-x_{n,2}|}{s_n t_{n,2}}\leq C,
        \end{align*}
        implying in particular that,
        \begin{align*}
          2R_{n,1}\leq \frac{|x_{n,1}-x_{n,2}|}{s_n} \leq C t_{n,2},
        \end{align*}
        and $\frac{|x_{n,1}-x_{n,2}|}{s_n}\to +\infty$ and~$t_{n,2}\to +\infty$ as~$n\to +\infty$.

        Recall the functions~$u_{n,1}$ and~$u_{n,2}$ defined by~\eqref{eq:un def} and~\eqref{eq:un2 def}.
        Putting
        \begin{align*}
            z_{n,1,2}\coloneqq \frac{x_{n,1} -x_{n,2}}{s_n t_{n,2}},
        \end{align*}
        which is uniformly bounded, possibly along a subsequence we have $z_{n,1, 2}\to z_{\infty,1,2}$, and then
        \begin{align*}
            u_{n,2}(z_{n,1,2})= t_{n,2}^{\frac{2}{p-1}}\theta_n\parenthesis{ v_n(x_{n,1} )-1}= \parenthesis{\frac{t_{n,2}}{t_{n,1}}}^{\frac{2}{p-1}} u_{n,1}(0)\to +\infty
        \end{align*}
        since~$u_{n,1}(0)=\phi(0)>0$ and~$\frac{t_{n,2}}{t_{n,1}}\to +\infty$.
        This is a contradiction to the fact that~$u_{n,2}$ is bounded from above by~$\phi(0)$.
    \end{proof}

\medskip

\subsection{No Fading spike if there is already a spike of Type I/II}\label{subs:5.2}
Next we wish to rule out the Fading alternative in presence of a spike of Type I or II. Therefore we consider the case where $s_n t_{n,2}$ has a positive lower bound,
which means~$t_{n,2}$ diverges to~$+\infty$ very fast, so that $v_n(x_{n,2})$ is a Fading sequence with $x_{n,2}$ converging at $x^*$.
We will see that this is impossible
as far as \eqref{Htype} is satisfied.

\medskip

    Without loss of generality we may assume that
    \begin{align*}
        \frac{1}{s_n t_{n,2}}\to \rho\geq 0.
    \end{align*}
    Recall that~$v_n(x_{n,2})=\sup\braces{v_n(x)\mid x\in \Omega\setminus B_{R_{n,1}s_n t_{n,1}}(x_{n,1})}$ and
    \begin{align*}
        \theta_n (v_n(x_{n,2})-1)=\phi(0)t_{n,2}^{-\frac{2}{p-1}}=\phi(0)\frac{s_n^{\frac{2}{p-1}}}{(s_n t_{n,2})^{\frac{2}{p-1}}}= \phi(0)(\rho+ o(1)) s_n^{\frac{2}{p-1}},
    \end{align*}
    thus
    \begin{align}\label{eq:decay away from x1}
        \sup_{\Omega\setminus B_{R_{n,1}s_n t_{n,1}}(x_{n,1})} v_n
        \leq v_n(x_{n,2}) = 1 + \phi(0)(\rho+o(1))\frac{s_n^{\frac{2}{p-1}}}{\theta_n}= 1 + \phi(0)(\rho+o(1))\eps_n^{\frac{2}{p-1}}.
    \end{align}
    Let~$G(x,x')$ be the Green function for the domain~$\Omega$ with Dirichlet boundary condition, and assume~$v_n=0$ on~$\p\Omega$.
    Then for any~$x\in \Omega$ we have
    \begin{align*}
        v_n(x)=\frac{1}{\eps_n^2}\int\limits_{\Omega} G(x,x') [v_n(x')-1]_+^p\dd{x'}.
    \end{align*}
    To estimate the value at a point~$x\in \Omega\setminus B_{r}(x_{1})$ where~$r>0$ is small,
    we split the domain~$\Omega$ into two parts:~$\Omega= B_{R_n s_n t_{n,1}}(x_{n,1})\cup (\Omega\setminus B_{R_n s_n t_{n,1}}(x_{n,1}))$.
    Therefore, near~$x_1$ we have:
    \begin{align*}
        \frac{1}{\eps_n^2} \int\limits_{B_{R_{n,1}s_n t_{n,1}}(x_{n,1})} G(x,x')[v_n(x')-1]_+^p\dd{x'}
        =&\frac{1}{\eps_n^2} \frac{(s_n t_{n,1})^2}{\theta_n^p t_{n,1}^{\frac{2p}{p-1}}}
        \int\limits_{B_{R_{n,1}}(0)} G(x, x_{n,1}+ s_n t_{n,1} z) [u_{n,1}(z)]_+^p \dd{z}\\
        =&\frac{1}{\theta_n t_{n,1}^{\frac{2}{p-1}}}\int\limits_{B_{R_{n,1} } (0)} G(x,x_{n,1}+s_n t_{n,1}z) [u_{n,1}(z)]_+^p\dd{z}.
    \end{align*}
    Since~$R_{n,1}\to+\infty$,  the balls $B_{R_n s_n t_{n,1}}(x_{n,1})$ exhaust the plane and~$u_{n,1}\to w^*$
    in~$C^2_{\rm loc}(\R^2)$, thus~$[u_{n,1}]_+^p\to [w^*]_+^p$ in~$C^0_{\rm loc}(\R^2)$ and they have uniformly bounded support.
    Observing that $s_n t_{n,1}\to 0^+$, we deduce that,
    \begin{align*}
        \frac{\theta_n}{\eps_n^2}\int\limits_{B_{R_{n,1}s_n t_{n,1}}(x_{n,1})} G(x,x')[v_n(x')-1]_+^p\dd{x'}
        =&\frac{1}{t_{n,1}^{\frac{2}{p-1}}} \int\limits_{B_1(0)} G(x,x_{1}) [w^*(z)]_+^p\dd{z} + o(1) \\
        =&\frac{I_p}{t_{n,1}^{\frac{2}{p-1}}}G(x,x_1)+ o(1).
    \end{align*}
    The integral on $\Omega\setminus B_{R_n s_n t_{n,1}}(x_{n,1})$ takes the form,
    \begin{align*}
        0\leq &\frac{\theta_n}{\eps_n^2}\int\limits_{\Omega\setminus B_{R_{n,1} s_n t_{n,1}}(x_{n,1})} G(x,x')[v_n(x')-1]_+^p\dd{x'} \\
        \leq&\frac{\theta_n}{\eps_n^2}\int\limits_{\Omega\setminus B_{R_{n,1} s_n t_{n,1}}(x_{n,1})} G(x,x')(\phi(0)(\rho+o(1)) \eps_n^{\frac{2}{p-1}})^p \dd{x'} \\
        \leq & s_n^{\frac{2}{p-1}}\phi(0)^p (\rho+o(1))^p \parenthesis{\int\limits_{\Omega}G(x,x')\dd{x'}}.
    \end{align*}
    Thus for~$x\in \Omega\setminus B_r(x^*)$,
    \begin{align*}
        \theta_n v_n(x)= \frac{I_p}{t_{n,1}^{\frac{2}{p-1}}}G(x,x^*)+ o(1), \qquad \mbox{as} \quad n\to+\infty.
    \end{align*}
    Since the right hand side is bounded while~$\theta_n\to+\infty$ and~$t_{n,1}$ are bounded from below, we conclude that for any~$r>0$ and any~
    $x\in \Omega\setminus B_r(x^*)$, ~$v_n(x)\to 0$ locally uniformly.
    In particular, fix~$r\in (0, \frac{1}{2}\dist(x^*,\p\Omega))$, there exists~$n_r\in\mathbb{N}$ such that
    \begin{align}\label{eq:vanish away from x1}
        [v_n(x)-1]_+ =0,   \qquad  \forall x \in \Omega\setminus B_r(x^*), \quad \forall n>n_r.
    \end{align}
    Remark that this estimates holds true under the assumption that~$x_{n,2}$ is a Fading sequence. By no means we can assume it to hold in general.
     In particular we cannot claim that there is only one spike in~$\Omega$.

    \

    Next we apply again the Green representation argument to see that, in view of \eqref{eq:decay away from x1} and \eqref{eq:vanish away from x1},
    then~$v_n(x_{n,2})\leq 1$ which contradicts the assumption \eqref{Hypxn2}.

    \

    Indeed, because of~\eqref{eq:vanish away from x1}, the Green representation formula above reduces to an
    integration over~$B_r(x^*)$ for~$n> n_r$: for any~$x\in \Omega$,
    \begin{align*}
        v_n(x)=\frac{1}{\eps_n^2}\int\limits_{\Omega} G(x,x')[v_n(x')-1]_+^p\dd{x'}
        =\frac{1}{\eps_n^2}\int\limits_{B_r(x^*)} G(x,x')[v_n(x')-1]_+^p\dd{x'}.
    \end{align*}
    Combining this identity with~$t_{n,1}^{\frac{2}{p-1}}\theta_n (v_n(x_{n,1})-1)=u_{n,1}(0)=\phi(0)$, we have
    \begin{align*}
        \theta_n v_n(x_{n,1})
        = \frac{\theta_n}{\eps_n^2}\int\limits_{B_r(x^*)} G(x_{n,1},x')[v_n(x')-1]_+^p\dd{x'}
        =u_{n,1}(0)t_{n,1}^{-\frac{2}{p-1}}+\theta_n.
    \end{align*}
    Therefore, by using the classical decomposition $G(x,x')=-\frac{1}{2\pi}\ln|x-x'| + H(x,x')$ where~$H(x,x')$ is the regular part (hence locally bounded),
    we have
    \begin{align*}
        \theta_n (v_n(x)- 1)
        =&\frac{\theta_n}{\eps_n^2} \int\limits_{B_r(x^*)} \parenthesis{G(x,x')-G(x_{n,1},x')}[v_n(x')-1]_+^p\dd{x'} + u_{n,1}(0)t_{n,1}^{-\frac{2}{p-1}} \\
        =&\frac{\theta_n}{\eps_n^2} \int\limits_{B_r(x^*)} \parenthesis{H(x,x')-H(x_{n,1},x')}[v_n(x')-1]_+^p\dd{x'} + u_{n,1}(0)t_{n,1}^{-\frac{2}{p-1}} \\
        & \qquad+ \frac{\theta_n}{\eps_n^2} \int\limits_{B_r(x^*)} \frac{1}{2\pi}\ln\frac{|x_{n,1}-x'|}{|x-x'|} [v_n(x')-1]_+^p\dd{x'} .
    \end{align*}
    Let~$t_{n,1}\to t_{\infty,1} \in [T_0, +\infty]$, then the term~$u_{n,1}(0)t_{n,1}^{-\frac{2}{p-1}}\to \phi(0)t_{\infty,1}^{-\frac{2}{p-1}}\geq 0$, which is in particular bounded.
    The integration involving~$H(x,x')-H(x_{n,1},x')$ is also bounded because of~\eqref{Hypo on p}.
    For the remaining part involving the fundamental solutions, we have that
    \begin{align*}
        &\frac{\theta_n}{\eps_n^2} \int\limits_{B_{R_{n,1}s_n t_{n,1}}(x_{n,1})} \frac{1}{2\pi}\ln\frac{|x_{n,1}-x'|}{|x-x'|} [v_n(x')-1]_+^p\dd{x'}\\
        =&\frac{1}{2\pi}\int\limits_{B_{R_{n,1} t_{n,1}} (0) } \ln\frac{|y|}{|\frac{x-x_{n,1}}{s_n} -y|} [\tilde{v}_n(y)]_+^p\dd{y}\\
        =&\frac{t_{n,1}^{-\frac{2}{p-1}} }{2\pi}\int\limits_{B_{R_{n,1} /t_{n,1} } (0)} \ln \frac{|z|}{|\frac{x-x_{n,1}}{s_n t_{n,1}} - z | }
        [u_{n,1}(z)]_+^p\dd z\\
        =&\frac{t_{n,1}^{-\frac{2}{p-1}} }{2\pi}\int\limits_{B_2 (0)} \ln \frac{|z|}{|\frac{x-x_{n,1}}{s_n t_{n,1}} - z | } [u_{n,1}(z)]_+^p\dd z
    \end{align*}
    meanwhile, in view of \eqref{eq:decay away from x1},
    \begin{align*}
        & \left| \frac{\theta_n}{\eps_n^2} \int\limits_{B_r(x_1)\setminus B_{R_{n,1} s_n t_{n,1}}(x_{n,1}) } \frac{1}{2\pi}\ln\frac{|x_{n,1}-x'|}{|x-x'|} [v_n(x')-1]_+^p\dd{x'}\right| \\
        \leq& \frac{\phi(0)^p(\rho+o(1))^p}{2\pi} s_n^{\frac{2}{p-1}} \int\limits_{\Omega\setminus B_{R_{n,1}s_n t_{n,1}}(x_{n,1})} \left|\ln\frac{|x_{n,1} -x'| }{|x-x'|} \right|\dd{x'}
        \leq C s_n^{\frac{2}{p-1}}.
    \end{align*}
    Thus, putting $x=x_{n,1}+s_n t_{n,1}z$, and~$ 2R_{n,1}  \leq |z| \leq \frac{\dist(x_1,\p\Omega)/2}{s_n t_{n,1}}$, we have that,
    \begin{align*}
        &u_{n,1}(z)=t_{n,1}^{\frac{2}{p-1}}\theta_n(v_n(x_{n,1}+ s_n t_{n,1} z)-1)
        = O(1)+\frac{1}{2\pi} \int\limits_{B_2 (0)} \ln \frac{|z'|}{|z - z' | } [u_{n,1}(z')]_+^p\dd z' =\\
        & O(1)+ \frac{1}{2\pi}\ln\frac{1}{|z|} \int\limits_{B_1(0)} \ln|z'| [w^*(z')]_+^p\dd{z'}.
    \end{align*}
    This fact immediately implies that,
    \begin{align*}
        [v_n(x)-1]_+^p \leq 0, \qquad \mbox{ in } B_r(x^*)\setminus B_{2R_n s_n t_{n,1}(x_{n,1})},
    \end{align*}
    which contradicts the assumption \eqref{Hypxn2}, that is, for some~$x_{n,2}\in B_r(x_1)\setminus B_{2R_n s_n(x_{n,1})}$,~$v_n(x_{n,2})-1>0$ for any $n$.
    Therefore, as far as \eqref{Htype} is satisfied, there cannot be a Fading spike at $x^*$, as claimed.

\subsection{Formation of another spike}\label{subs:5.3}
Because of \eqref{Htype} and since we have ruled out the Fading spike at $x^*$,
we necessarily have $s_n t_{n,2}\to 0^+$ together with \eqref{eq:far I}. Therefore, for any~$R\geq 1$ and for~$n$ large,
    \begin{align}\label{eq:5.3}
        B_{R s_n t_{n,2}} (x_{n,2}) \cap B_{R s_n t_{n,1}} (x_{n,1})=\emptyset.
    \end{align}
    In particular the rescaled functions~$u_{n,2}$ defined in~\eqref{eq:un def} assume their global maximum at the origin:~
    $\sup\limits_{B_R(0)} u_{n,2}=u_{n,2}(0)=\phi(0)$ for any~$R>0$ and for~$n$ large. Moreover we have the following
    uniform bound about the $(p)$-mass:
    \begin{align*}
        &\int\limits_{B_R(0)} [u_{n,2}]_+^p\dd{z}
        =\frac{t_{n,2}^{\frac{2p}{p-1}}\theta_n^p }{s_n^2 t_{n,2}^2} \int\limits_{B_{R s_n t_{n,2}}(x_{n,2})} [v_n(x)-1]_+^p\dd{x}=\\
        &\frac{t_{n,2}^{\frac{2}{p-1}}\theta_n^p }{s_n^2}\int\limits_{B_{R s_n t_{n,2}}(x_{n,2})} [v_n(x)-1]_+^p\dd{x}=
        \frac{t_{n,2}^{\frac{2}{p-1}}\theta_n }{\eps_n^2}\int\limits_{B_{R s_n t_{n,2}}(x_{n,2})} [v_n(x)-1]_+^p\dd{x}\leq \phi(0)H_{p-1},
    \end{align*}
    where we used \eqref{2ways} and the last inequality follows as in \eqref{intboundp}.\\
    At this point, according to Lemma~\ref{lemma:ALT by Harnack}, up to a subsequence $u_{n,2}$ converges to an entire
    solution~$u_{\infty,2}$ of the form~\eqref{eq:entire solution} with~$R_p=1$ (since~$u_{n,2}(0)=\phi(0)$), i.e.~$u_{\infty,2}=w^*$.
    Thus the sequence~$v_n(x_{n,2})$ yields a second spike either of Type I or of Type II.

    Remark that~$v_n(x_{n,2})-1\leq v_n(x_{n,1})-1$, whence we surely have that $t_{n,2}\geq t_{n,1}$,
    but this is not enough to determine the Type of the second spike arising from~$v_n(x_{n,2})$.
    Indeed, if~$v_n(x_{n,1})$ yields a spike of Type I, then in principle the second spike could be either of Type I or of Type II,
    while if $v_n(x_{n,1})$ is already of Type II, then the second spike must be of Type II as well.

\bigskip

\section{The proof of Theorem \ref{thm:ALT many maximum}}\label{sec6}
In this section we prove Theorem \ref{thm:ALT many maximum}.
\begin{proof}[The Proof of Theorem \ref{thm:ALT many maximum}]$\,$\\
First of all we have the following
\begin{lemma}\label{lem:bdy} Let~$v_n$ be a sequence of solutions of~\eqref{eq:blow up eqn}. There exists $d_*>0$ such that there are no critical points of $v_n$ in
$\om_{*}=\{x\in\om\,:\,\mbox{dist}(x,\pa \om)<d_*\}$.
\end{lemma}
\begin{proof}[The Proof of Lemma \ref{lem:bdy}]$\,$\\ Since we are in dimension $d=2$ and since by the regularity assumption about the
the domain we have that $\pa \Omega$ satisfies a uniform exterior ball condition, then the proof is a well known consequence of a
moving plane argument (\cite{gnn}) to be combined with a Kelvin transform. We refer the reader to Proposition 4 in \cite{MaW} for further details.
\end{proof}

\medskip
{\it Proof of} (a)-(b)-(c)-(d).\\
Let $v_n$ be a sequence of solutions of~\eqref{eq:blow up eqn} and let
$$
v_n(x_{n,1})=\sup\limits_{\Omega}{v_n},
$$
then, possibly along a subsequence, we have $x_{n,1}\to x_1\in\om$, where we used Lemma \ref{lem:bdy}. According to Theorem \ref{thm:ALT near maximum} in principle
we could have the Vanishing alternative, which is easily ruled out. Indeed, if this was the case, we could peak any open and relatively
compact set $\om_1\Subset \om$ such that $x_1\in \om_1$ to deduce that,
$$
\sup\limits_{\om}[v_n-1]_+=[v_n(x_{n,1})-1]_+=\sup\limits_{\om_1}[v_n-1]_+=0,
$$
for any $n$ large enough, which obviously contradicts \eqref{eq:NV p}. Thus the Vanishing alternative cannot happen and the next step will be to rule out the Fading spike alternative.\\
If by contradiction this was the case, by Theorem \ref{thm:ALT near maximum} we would have that,
$$
v_n(x)- 1\leq  \frac{\eps_n^{\frac{2}{p-1}}}{(s_n t_n)^{\frac{2}{p-1}}}(\phi(0)+o(1)), \quad \forall |x-x^*|\leq \dt,
$$
for some $\dt$ small, but since $v_n$ is the maximum in $\om$, then this local estimate would hold for any in $x\in \om$, and the estimate
\eqref{eq:Fading energies} would take the form,
$$
\frac{\theta_n}{\eps_n^2}\int\limits_{\om} [v_n-1]_+^{p}\dd{x}\leq C s_n^{\frac{2}{p-1}}\to 0,
$$
which contradicts \eqref{eq:NV p}. Therefore we infer again from Theorem \ref{thm:ALT near maximum} that we have either
a spike of Type I or of Type II.

\

    Inductively, let us assume that for some $m\geq 2$ we have already chosen $m-1$ sequences of local maximizers
    $x_{n,j}$ with $v_n(x_{n,j})>1$ and~$x_{n,j}\to x_{\infty,j}\in\om$ as~$n\to+\infty$ for~$j=1,2,\cdots,m-1$, which yield either
    a spike of Type I or of Type II each. Remark that, in view of Lemma \ref{lem:bdy}, for each $j$ we have $x_{\infty,j}\in\om$.
    Thus, according to Theorem \ref{thm:ALT near maximum}, we also have well defined
    sequences~$R_{n,j}$ and~$t_{n,j}$, for~$j=1,2,\cdots, m-1$ satisfying the properties listed either in (B-i) or in (B-ii).\\
    At this point, let us define,
    $$
    x_{n,m}\in\Omega\setminus \bigcup_{j=1}^{m-1} B_{R_{n,j}s_n t_{n,j}}(x_{n,j})
    $$
    such that,
    \begin{align*}
        v_n(x_{n,m})=\sup\braces{v_n(x)\mid x\in \Omega\setminus \bigcup_{j=1}^{m-1} B_{R_{n,j}s_n t_{n,j}}(x_{n,j})}
    \end{align*}
    If~$v_n(x_{n,m})\leq 1$ for infinitely many~$n$, then passing to a subsequence, we are done with the spike analysis for this subsequence.
    In fact we define,
    $$
    R_{n}:=\min\limits_{j\in \{1,\cdots,m-1\}}R_{n,j},
    $$
    which obviously satisfies $R_n s_nt_{n,j}\to +\ii$, for any $j$.
    Then, observing that, in view of \eqref{eq:NV p}, at least one sequence $x_{n,j}$ must yield a Spike of Type I, we can set $N_I+N_{II}=m-1$, with
    $N_{I}\geq 1$ and define, as in the statement of the Theorem, $X_{\rm I}=\{x^*_{n,i}\}_{i\in\{1,\cdots,N_I\}, n\in\N}$
    to be the sequences of local maximizers yielding Type I spikes and $X_{\rm II}=\{x^{**}_{n,i}\}_{i\in\{1,\cdots,N_{II}\},n\in\N}$
    to be the sequences of local maximizers yielding Type II spikes.
    As a consequence, according to Remark \ref{rem:4.6}, we have that,
     $$
     v_n(x)= 1+\frac{1}{t_{n,j}^{\frac{2}{p-1}}\theta_n} w^*(\frac{x-x_{n,j}}{s_n t_{n,j}})+ o(\frac{1}{t_{n,j}^{\frac{2}{p-1}}\theta_n }),
     \quad |x-x_{n,j}|\leq R_{n}s_nt_{n,j},
     $$
     where
     $$
     t_{n,j}\to t_{\ii,j}\in (T_0,+\ii] \mbox{ as }n\to+\infty \mbox{ and }
     \graf{t_{\ii,j}\in (T_0,+\ii), \;\mbox{if}\; x_{n,j}=x^*_{n,i},\,i\in \{1,\cdots,N_{I}\}\\\,\\
     t_{\ii,j}=+\ii, \;\mbox{if}\; x_{n,j}=x^{**}_{n,i},\,i\in \{1,\cdots,N_{II}\}}
     $$
     so that, putting,
     $$
     D^*_n:=\left\{\bigcup_{i=1}^{N_I} B_{R_{n}s_n t_{n,i}}(x^*_{n,i})\right\},\quad
     D^{**}_n:=\left\{\bigcup_{i=1}^{N_{II}} B_{R_{n}s_n t_{n,i}}(x^{**}_{n,i})\right\}
     $$

     $$
     D_n:=D^*_n\bigcup D^{**}_n
     $$
     we also have that,
           \begin{align*}
              \frac{1}{\eps_n^2}\int\limits_{\om} [v_n(x)-1]_+^{p-1}\dd{x}=\frac{1}{\eps_n^2}\int\limits_{D_n} [v_n(x)-1]_+^{p-1}\dd{x}=
              (N_I+N_{II})I_{p-1} + o(1)
            \end{align*}
            \begin{align*}
              \frac{\theta_n}{\eps_n^2}\int\limits_{\om} [v_n(x)-1]_+^{p}\dd{x}=
              \frac{\theta_n}{\eps_n^2}\int\limits_{D^*_n} [v_n(x)-1]_+^{p}\dd{x}=
              \sum\limits_{i\in \{1,\cdots,N_I\}}\frac{I_p}{t_{\ii,i}^{\frac{2}{p-1}}} + o(1),
            \end{align*}
            as~$n\to+\infty$. It is worth to remark that the singular set $\Sigma=\Sigma_I\cup \Sigma_{II}$ in the statement is just the set of cluster points
            of $X_I$ and $X_{II}$, implying in particular that
            $$
            D_n\Subset (\Sigma)_r\quad \mbox{for any $r$ small enough}.
            $$

    Moreover, since $v_n$ is harmonic in $\om\setminus D_n$, by the maximum principle we have that,
    $$
    t_{n,j}^{\frac{2}{p-1}}\theta_n (v_n(x)-1)\leq \phi^{'}(1)\log(R_n),\quad \forall\,x \in \om\setminus D_n, \forall \,j\in\{1,\cdots,N_I+N_{II}\},
    $$
    implying that $[v_n-1]_+=0$ in  $\om\setminus D_n$. Therefore, recalling (B-i) and (B-ii), by the Green representation formula we see that,
    \begin{align} \nonumber
  \theta_n  v_n(x)= &\frac{\theta_n}{\eps_n^2}\int\limits_{D_n}G(x,x')[v_n(x')-1]^p_+\dd{x'}\\ \nonumber
    =&\int\limits_{D^*_n}G(x,x')\frac{\theta_n}{\eps_n^2}[v_n(x')-1]^p_+\dd{x'}+
   \int\limits_{D^{**}_n}G(x,x')[v_n(x')-1]^p_+\dd{x'}\\ \nonumber
   =& \sum_{i=1}^{N_I} \frac{\theta_n}{\eps_n^2} \frac{(s_nt_{n,i})^2}{\theta_n^p t_{n,i}^{\frac{2p}{p-1}}} \int\limits_{B_{R_{n}}}
   G(x,x^*_{n,i}+ s_n t_{n,i}z)[u_n(z)]_+^p\dd{z}\\\nonumber
   & +\sum_{i=1}^{N_{II}} \frac{\theta_n}{\eps_n^2} \frac{(s_nt_{n,i})^2}{\theta_n^p t_{n,i}^{\frac{2p}{p-1}}} \int\limits_{B_{R_{n}}}
   G(x,x^{**}_{n,i}+ s_n t_{n,i}z)[u_n(z)]_+^p\dd{z}\\\nonumber
   =& \sum_{i=1}^{N_I}\frac{1}{t_{n,i}^{\frac{2}{p-1}}}\int\limits_{B_{2}}
   G(x,x^*_{n,i}+ s_n t_{n,i}z)[u_n(z)]_+^p\dd{z}+\sum_{i=1}^{N_{II}} \frac{1}{t_{n,i}^{\frac{2}{p-1}}}\int\limits_{B_{2}}
   G(x,x^{**}_{n,i}+ s_n t_{n,i}z)[u_n(z)]_+^p\dd{z}\\ \label{green:1}
   =&\sum_{i=1}^{N_I} \frac{I_p}{t_{n,i}^{\frac{2}{p-1}}}G(x,x^*_{n,i})+o_r(1)=\sum_{i=1}^{N_I} \frac{I_p}{t_{n,i}^{\frac{2}{p-1}}}G(x,x^*_{\ii,i})+o_r(1),\qquad \forall\, x\in \om\setminus (\Sigma)_r
    \end{align}
   where $o_r(1)$ is some quantity which uniformly converges to $0$ for any fixed $r$ small enough. Therefore the properties (a)-(b)-(c) in the claim
   would be proved as far as $v_n(x_{n,m})\leq 1$ along a subsequence, whence we assume w.l.o.g. that $v_n(x_{n,m})>1$ for all~$n\geq 1$ and
   let~
   $$
   t_{n,m}\coloneqq \parenthesis{\frac{\phi(0)}{\theta_n (v_n(x_{n,m})-1)}}^{\frac{p-1}{2}}.
   $$

   We can assume that $x_{n,m}\to x_{\ii,j}$ for some $j\in \{1,\cdots,m-1\}$ otherwise the proof is easier.
   By the result in subsection \ref{subs:5.2}, $x_{n,m}$ cannot be a Fading spike while by the results in subsections \ref{subs:5.1} and \ref{subs:5.3}
   we have that,
    \begin{align*}
        \frac{|x_{n,j}-x_{n,m}|}{s_n t_{n,m}} \to +\infty, \qquad \mbox{ for each }\; j=1,2,\cdots, m-1,
    \end{align*}
   and $x_{n,m}$ yields a Type I or a Type II spike, whence according to (B-i), (B-ii), $t_{n,m}>T_0$, and in particular as in \eqref{eq:5.3},
  \begin{align*}
        B_{R s_n t_{n,m}} (x_{n,m}) \cap B_{R_{n,j} s_n t_{n,j}} (x_{n,j})=\emptyset\quad \forall\,j=1,\cdots, m-1.
    \end{align*}
Possibly along a subsequence, we can find $R_{n,m}\to+\infty$ and~$t_{n,m}\in (T_0,+\infty]$ such that
    \begin{align*}
        R_{n,m} s_n t_{n,m}\to 0^+.
    \end{align*}
     Thus the rescaled functions $u_{n,m}$, defined as in \ref{eq:un def}, satisfy,
    \begin{align*}
        -\Delta u_{n,m}(z) = [u_{n,m}(z)]_+^p, \qquad \forall z\in B_{R_{n,m}}(0)
    \end{align*}
    again with uniformly bounded~$(p-1)$ and $(p)$-masses. Therefore $u_{n,m}$ converges in~$C^2_{\rm loc}(\R^2)$ to some
    $u_{\infty,m}$ which is an entire solution of~\eqref{eq:entire eqn} with~$u_{\infty,m}(0)=\phi(0)$, hence~$u_{\infty,m}=w^*$
    and this gives us the~$m$-th spike. Remark that putting,
    putting
    $$
     D_n:=\left\{\bigcup_{j=1}^{m} B_{R_{n,j}s_n t_{n,j}}(x_{n,j})\right\},
     $$
    we have that
\begin{align*}
              \frac{1}{\eps_n^2}\int\limits_{\om} [v_n(x)-1]_+^{p-1}\dd{x}=\frac{1}{\eps_n^2}\int\limits_{D_n} [v_n(x)-1]_+^{p-1}\dd{x}=
              mI_{p-1} + o(1),
            \end{align*}
            implying that, due to \eqref{Hypo on p-1}, the induction argument has to stop after a finite number of steps.
            In particular if $m$ were the total number of spikes of Type I or II, then we would have,
              \begin{align*}
            \#(\Sigma) \leq \frac{ H_{p-1}}{m I_{p-1}}.
        \end{align*}

Let $m$ the total number of spikes of Type I or II, then it is readily seen that (a)-(b)-(c) follow as above with $N_{I}+N_{II}=m$ and $N_I\geq 1$.
We skip the details to avoid repetitions.\\

            At last, observe that, according to (B-i) and (B-ii) in Theorem \ref{thm:ALT near maximum}, the plasma region, that is the subset
            $$
            \Omega^{n,+}\coloneqq\braces{x\in\Omega\mid v_n(x)>1}
            $$
            consists of \un{asymptotically round points} in the sense of Caffarelli--Friedman (\cite{CF}), namely, for any~$0<\theta<1$
    \begin{align*}
        \bigcup_{j=1}^{m} B_{(1-\theta) s_n t_{n,j}}(x_{n,j})
        \Subset \Omega^{n,+}
         \Subset\bigcup_{j=1}^{m} B_{(1+\theta) s_n t_{n,j}}(x_{n,j})
    \end{align*}
    for any $n$ sufficiently large. This fact concludes the proof of (a)-(b)-(c)-(d).\\

The global behavior of the possible spikes are sketched in the Figure~\ref{fig:2}.

\begin{figure}
  \centering
  \includegraphics[width=0.8\textwidth]{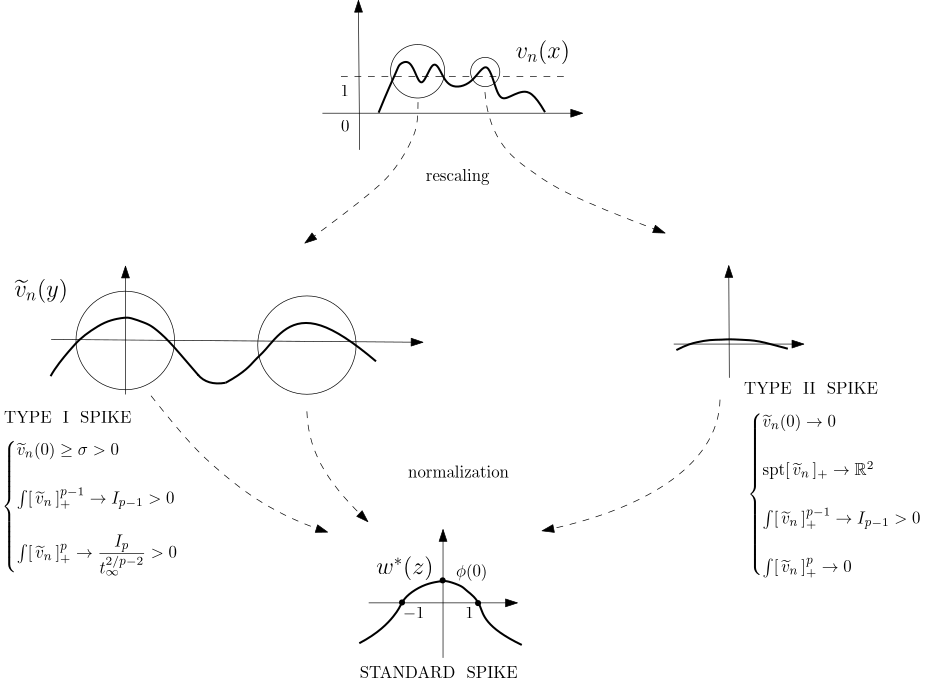}
  \caption{Global behavior of the spikes} \label{fig:2}
\end{figure}

\


We are just left with the proof of (e).\\
{\it Proof of} (e).\\
  We recall \eqref{green:1}, which we write as follows,
    \begin{align*}
       \lim_{n\to+\infty} \theta_n v_n(x)
         =&\sum_{i=1}^{N_I} \frac{I_p}{t_{\infty,i}^{\frac{2}{p-1}}} G(x,x^*_{\infty,i}) =
         I_p\sum_{\ell =1}^{m_1} M_{\ell}
         G(x,x^*_{\infty,\ell}) =: I_p\mathcal{G}(x),
    \end{align*}
    where $M_\ell$ was defined in \eqref{p-mass}.
    This convergence is uniform in $\Omega\setminus (\Sigma)_r$, where (see (b)), $v_n(x)$ is of order O$(\frac{1}{\theta_n})$. However it
    is well known (see for example \cite{MaW} or either \cite{BJW}) that in this situation a careful analysis of the Pohozaev identity yields a constraint
    about $(x^*_{\ii,1},\cdots,x^*_{\ii,m_1})$. Indeed, let~$H$ be the regular part of the Green's function:
    \begin{align*}
        G(x',x'')=\frac{1}{2\pi}\ln\frac{1}{|x'-x''|} + H(x',x''), \quad \forall x',x''\in\Omega, x'\neq x'',
    \end{align*}

The functions~$\theta_n v_n\in C^{2,\beta}(\Omega)$ satisfy the equations
        \begin{align*}
            -\Delta (\theta_n v_n)=\frac{1}{s_n^2} [\theta_n v_n -\theta_n]_+^p.
        \end{align*}
     By using $\nabla (\theta_n v_n)$ as test functions in~$\Omega'\Subset \Omega$ (with~$\p\Omega'$ smooth), we obtain the vectorial Pohozaev identity:
        \begin{align*}
            \int\limits_{\p\Omega'} -\p_\nu (\theta_n v_n) \nabla (\theta_n v_n) +\frac{1}{2}|\nabla (\theta_n v_n)|^2 \nu \dd{s}
            =\int\limits_{\p\Omega'} \frac{1}{p+1}\frac{1}{s_n^2} [\theta_n v_n -\theta_n ]_+^{p+1} \nu \dd{s}.
        \end{align*}
        Now let $a_1:=x^{*}_{\ii,1}$ and peak and $r>0$ small enough such that~$\Omega'=B_{2r}(a_1)$ does not contain any other spike point of Type I.
        Passing to the limit~$n\to+\infty$, the right hand side is readily seen to vanish, meanwhile the left hand side will converge to the
        corresponding integral with~$\theta_n v_n$ replaced by~$I_p\mathcal{G}$, whence we have that,
        \begin{align}\label{eq:vector Pohozaev}
            \int\limits_{\p B_r(a_1)} -\p_\nu\mathcal{G} \nabla\mathcal{G}+ \frac{1}{2}|\nabla\mathcal{G}|^2\nu \dd{s}=0.
        \end{align}

      At this point observe that,
        \begin{align*}
            \mathcal{G}(x)
            =&\sum_{i=1,\cdots N_I\;:\; x^*_{\infty, i}=a_1}  \frac{1}{t_{\infty,i}^{\frac{2}{p-1}}} G(x,x_{\infty,i})
             +\sum_{i=1,\cdots N_I\;:\; x^*_{\infty, i}\neq a_1} \frac{1}{t_{\infty,i}^{\frac{2}{p-1}}} G(x,x_{\infty,i})  \\
            =&\sum_{i=1,\cdots N_I\;:\; x^*_{\infty, i}=a_1}
            \left(\frac{1}{2\pi}\frac{1}{t_{\infty,i}^{\frac{2}{p-1}}} \ln\frac{1}{|x-a_1|} + \frac{1}{t_{\infty,i}^{\frac{2}{p-1}}}  H(x,a_1)\right)
             +\sum_{i=1,\cdots N_I\;:\; x^*_{\infty, i}\neq a_1} \frac{1}{t_{\infty,i}^{\frac{2}{p-1}}} G(x,x_{\infty,i}) \\
             =&\frac{M_1}{2\pi} \ln\frac{1}{|x-a_1|}
              + \parenthesis{ M_1  H(x,a_1)
             +\sum_{i=1,\cdots N_I\;:\; x^*_{\infty, i}\neq a_1} \frac{1}{t_{\infty,i}^{\frac{2}{p-1}}} G(x,x_{\infty,i}) } \\
             =&\frac{M_1}{2\pi} \ln\frac{1}{|x-a_1|} + F_1(x)
        \end{align*}
        where
        \begin{align*}
            F_1(x)\equiv M_1 H(x,a_1) + \sum_{i=2}^{m_1} M_i G(x,a_i).
        \end{align*}

         Inserting these expression of $\mathcal{G}(x)$ into~\eqref{eq:vector Pohozaev} and letting~$r\to 0^+$,
         exactly the same computations in either \cite{MaW} (see also \cite{BJW}) show that,
        \begin{align*}
            \nabla F_1(a_1)=0,
        \end{align*}
        which is the same as to say that,
        \begin{align*}
            \nabla_1 \mathcal{H}(x^*_{\infty, 1},x^*_{\infty, 2},\cdots, x^*_{\infty, m_1})=0.
        \end{align*}
        The vanishing of the other derivatives readily follows by a permutation of the variables.
\end{proof}

\bigskip

\section{The Proof of Theorem \ref{thm:last}}\label{sec7}
In this section we prove Theorem \ref{thm:last}.
\begin{proof}[{The Proof of Theorem \ref{thm:last}}]
Putting $v_n=\frac{\lm_n}{|\al_n|}\psi_n$ and $\eps^2_n=(|\al_n|^{p-1}\lm_n)^{-1}$, because of $\lm_n\to +\ii$ and $|\al_n|\geq 1$,
it is readily seen that $\eps_n\to 0$ and $v_n$ is a solution of
\eqref{eq:blow up eqn} where, in view of \eqref{H4},
$$
\int\limits_\Omega \frac{1}{\eps_n^2}[v_n-1]_+^{p-1}\dd{x}= \lm_n\int\limits_\Omega[\al_n+\lm_n\psi_n]_+^{p-1}\leq C_{p-1},
$$
whence \eqref{Hypo on p-1} is satisfied. Because of \eqref{H3} and $\lm_n\to +\ii$, $|\al_n|\geq 1$ we have that
$|\al_n|\to +\ii$ and
\begin{align}\label{part-Hp}
\int\limits_\Omega \frac{\theta_n}{\eps_n^2}[v_n-1]_+^{p}\dd{x} =
\int\limits_\Omega \frac{\lm_n}{|\al_n|}\theta_n[\al_n+\lm_n\psi_n]_+^{p}=\frac{\lm_n}{|\al_n|}\theta_n,
\end{align}
where we recall that
$$
      \parenthesis{\frac{\eps_n}{s_n}}^{\frac{2}{p-1}}\theta_n\equiv  \parenthesis{\frac{\eps_n}{s_n}}^{\frac{2}{p-1}} \phi'(1)\ln(\sqrt{\pi} s_n)=1,
$$
and in particular that $s_n\to 0$ as $\eps_n\to 0$. Elementary arguments show that
$$
s_n^{\frac{2}{p-1}}=(1+o(1))\phi^{'}(1)\eps_n^{\frac{2}{p-1}}\log(\eps_n)
$$
and consequently that
\begin{align}\label{part-theta}
\theta_n=\phi^{'}(1)\ln(\sqrt{\pi} s_n)=(1+o(1))\phi^{'}(1)\log(\eps_n)=\frac{(1+o(1))}{2}|\phi^{'}(1)|\log(|\al_n|^{p-1}\lm_n).
\end{align}
Therefore we deduce from \eqref{part-Hp} and \eqref{H4} that,
$$
\int\limits_\Omega \frac{\theta_n}{\eps_n^2}[v_n-1]_+^{p}\dd{x}=\frac{\lm_n}{|\al_n|}\theta_n=
\frac{(1+o(1))}{2}\frac{\lm_n}{|\al_n|}|\phi^{'}(1)||\log(|\al_n|^{p-1}\lm_n)\graf{\leq |\phi^{'}(1)|C_p,\\ \geq \frac{|\phi^{'}(1)|}{2 C_p}},
$$
implying that \eqref{Hypo on p} and \eqref{eq:NV p} are both satisfied as well. As a consequence all the conclusions of Theorem \ref{thm:ALT many maximum} hold true for $v_n$
and in particular we deduce from \eqref{p-2:nquant} that,
\begin{align}\label{cii}
              \frac{\lm_n}{|\al_n|}\theta_n=\frac{\theta_n}{\eps_n^2}\int\limits_{\om} [v_n(x)-1]_+^{p}\dd{x}\to
              (c_\ii)^{-1}:=\sum\limits_{j\in \{1,\cdots,N_I\}}\frac{I_p}{t_{\ii,j}^{\frac{2}{p-1}}}.
\end{align}

Therefore from \eqref{part-theta} we have,
\begin{align*}
&\frac{|\al_n|}{\lm_n}=\frac{(1+o(1))}{2}|\phi^{'}(1)|\log(|\al_n|^{p-1}\lm_n)(1+o(1))c_\ii=\\
&(1+o(1))c_\ii|\phi^{'}(1)|\frac{p-1}{2}\left(\log(|\al_n|) +\log(\lm_n)\right),
\end{align*}
which readily implies that
$$
|\al_n|=(1+o(1))c_\ii|\phi^{'}(1)|\frac{p-1}{2}\lm_n\log(\lm_n)
$$
and consequently, again by \eqref{part-theta}
$$
\theta_n=(1+o(1))c_\ii|\phi^{'}(1)|\frac{p-1}{2}\log(|\al_n|).
$$
This fact together with \eqref{p-1:quant} concludes the proof of (i). The remaining part of the statement is just a rewriting of (a)-(b)
and (d)-(e) of Theorem \ref{thm:ALT many maximum} in terms of $\psi_n$, where in particular one uses \eqref{cii}.
\end{proof}

\end{document}